\documentclass[reqno, 12pt]{amsart}

\usepackage{array}
\usepackage{amsmath}
\usepackage{amsfonts}
\usepackage{amssymb}
\usepackage{enumerate}
\usepackage{amsthm}
\usepackage{amsmath, amscd}
\usepackage{caption}
\usepackage[usenames, dvipsnames]{color}
\usepackage{xy}
\xyoption{all}
\usepackage{hyperref}
\usepackage{tikz}
\usepackage{tikz-cd}
\usepackage{marginnote}
\usetikzlibrary{matrix,arrows,backgrounds}
\newtheorem{theorem}{Theorem}[section]
\newtheorem{lemma}[theorem]{Lemma}
\newtheorem{proposition}[theorem]{Proposition}
\newtheorem{corollary}[theorem]{Corollary}
\newtheorem{conjecture}[theorem]{Conjecture}

\theoremstyle{definition}
\newtheorem{definition}[theorem]{Definition}
\newtheorem{notation}[theorem]{Notation}
\newtheorem{remark}[theorem]{Remark}
\newtheorem{example}[theorem]{Example}
\newtheorem*{theorem*}{Theorem}
\newtheorem*{corollary*}{Corollary}
\newtheorem*{conjecture*}{Conjecture}
\newtheorem*{question*}{Question}
\numberwithin{equation}{section}
\newtheorem{question}[theorem]{Question}

\newcommand{\newterm}{\textsf}

\newcommand{\dbcoh}[1]{\operatorname{D}^{\operatorname{b}}(\operatorname{coh }#1)}
\newcommand{\dbmod}[1]{\operatorname{D}^{\operatorname{b}}(\operatorname{mod }#1)}

\newcommand{\dabs}[1]{\operatorname{D}^{\operatorname{abs}}[#1]}

\newcommand{\Hom}{\operatorname{Hom}}

\newcommand{\Z}{\mathbb{Z}}
\newcommand{\R}{\mathbb{R}}

\newcommand{\Q}{\mathbb{Q}}
\newcommand{\A}{\mathbb{A}}
\def\O{\mathcal{O}}
\def\P{\mathbb{P}}

\newcommand{\End}{\operatorname{End}}
\newcommand{\spec}{\operatorname{Spec}}
\newcommand{\Rhom}{\operatorname{RHom}}
\newcommand{\Rhomi}[1]{\operatorname{R}^{#1}\operatorname{Hom}}
\newcommand{\Rshom}{\operatorname{R}\mathcal{H}\mathit{om}}
\newcommand{\Ext}{\operatorname{Ext}}
\newcommand{\gldim}{\operatorname{gl}\dim}
\newcommand{\tot}{\operatorname{tot}}
\newcommand{\cone}{\operatorname{Cone}}
\newcommand{\E}{\mathcal{E}}
\newcommand{\Fact}[1]{\operatorname{Fact }(#1)}
\newcommand{\fact}[1]{\operatorname{fact }(#1)}
\newcommand{\conv}{\operatorname{Conv}}

\title[Towards NCCRs of affine toric Gorenstein varieties]{Towards non-commutative crepant resolutions of affine toric Gorenstein varieties}

\author{Aimeric Malter \textsuperscript{1}}
\email{aimericmalter@bimsa.cn}
\author{Artan Sheshmani \textsuperscript{1,2}}
\email{artan@mit.edu}
\address{\textsuperscript{1}Beijing Institute of Mathematical Sciences and Applications, No. 544, Hefangkou Village, Huaibei Town, Huairou District, Beijing 101408}
\address{\textsuperscript{2}Massachusetts Institute of Technology, IAiFi Institute, 77 Massachusetts Ave, 26-555. Cambridge, MA 02139}

\begin{document}

\bibliographystyle{alpha}
\begin{abstract}
In this paper we prove a common generalisation of results by \v{S}penko-Van den Bergh and Iyama-Wemyss that can be used to generate non-commutative crepant resolutions (NCCRs) of some affine toric Gorenstein varieties.
We use and generalise results by Novakovi\'{c} to study NCCRs for affine toric Gorenstein varieties associated to cones over polytopes with interior points. As a special case, we consider the case where the polytope is reflexive with $\le \dim P+2$ vertices, using results of Borisov and Hua to show the existence of NCCRs.
\end{abstract}
\maketitle

\section{Introduction}

The study of singularities and their associated resolutions has been a central topic of research in algebraic geometry. Oftentimes, resolving singularities will yield more complicated objects than one started with. For instance, resolving a singularity of an affine scheme rarely results in another affine scheme. Geometrically speaking, the approach to deal with this is to simply widen the scope of objects we consider. Algebraically, another approach has emerged, which is to consider non-commutative objects as resolutions. The underlying philosophy is to treat algebraic objects, oftentimes triangulated categories or sheaves of rings, as if they came from some underlying geometric object (even in the absence thereof). The following conjecture, presented independently by Bondal-Orlov \cite{BO02} and Kawamata \cite{Kaw02} at the ICM 2002, is central to this philosophy.
\begin{conjecture*}[\cite{BO02},\cite{Kaw02}]
    \label{Conj:BondalOrlovKawamata}
    Assume $X$ is a normal algebraic variety with Gorenstein singularities and $\pi:Y_i\rightarrow X$ for $i=1,2$ are two crepant resolutions (by schemes or DM-stacks). Then there is an equivalence of triangulated categories
   $F:\dbcoh{Y_1}\simeq \dbcoh{Y_2}$, linear over $X$.
\end{conjecture*}
Van den Bergh \cite{VdB3Dflops} considered the above conjecture for the case that $X$ has dimension 3 and $\pi_1, \pi_2$ form a flop. Assume, for simplicity, that $X=\spec R$ is affine. Then Van den Bergh considers tilting bundles $\mathcal{T}_i$ on $Y_i$, with corresponding endomorphism rings $\Lambda_i$. The Conjecture follows from proving the equivalences\[
\dbcoh{Y_1}\simeq\dbmod{\Lambda_1^\circ},\quad \dbcoh{Y_2}\simeq \dbmod{\Lambda_2}, \quad \Lambda_1^\circ\stackrel{\text{Morita}}{\simeq}\Lambda_2,
\]
where $(-)^\circ$ denotes right modules (and all other modules are left modules).
Thus, Van den Bergh finds that the two crepant resolutions $Y_1, Y_2$ are both in fact derived equivalent to the same non-commutative ring (which is either $\Lambda_1^\circ$ or $\Lambda_2$). This motivates the idea to think of this non-commutative ring as a third crepant resolution of $X$, albeit a non-commutative one. In \cite{VdB04}, Van den Bergh introduces the notion of \newterm{non-commutative (crepant) resolution}, abbreviated as NC(C)R, to generalise this phenomenon. 
Van den Bergh builds on the above Conjecture by Bondal-Orlov and Kawamata, and proposes the following.
\begin{conjecture*}[Conjecture 2.7 in \cite{VdB04}]
    \label{Conj:VdBdereq}
    All crepant resolutions of $X$ (commutative as well as non-commutative) are derived equivalent.
\end{conjecture*}
Given a normal Noetherian domain $R$, we call an $R$-algebra $\Lambda$ a non-commutative resolution of $R$ if $\Lambda=\End_R(M)$ for a non-zero finitely generated reflexive $R$-module $M$ such that $\operatorname{gl}\dim\Lambda<\infty$. If in addition $R$ is Gorenstein and $\Lambda$ is maximal Cohen-Macaulay as $R$-module, we call the resolution crepant. A lot of work has been done relating to NCCRs, ranging from constructions thereof (see e.g. work of \v{S}penko and Van den Bergh \cite{SVdB17}), to using the construction of NCCRs to understand deeper problems in mathematics (e.g. Rennemo and Segal \cite{RS19} using NCCRs to study the Pfaffian correspondence and construct window equivalences, establishing a form of homological projective duality proposed by Hori).   
 
However, the existence of NCCRs is not guaranteed, and it remains an interesting question as to when they exist. Even in some of the most natural cases to consider, the existence of NCCRs proved difficult to exhibit, and so the following Conjecture remains open.
\begin{conjecture*}
    An affine Gorenstein toric variety always has an NCCR.
\end{conjecture*}
This conjecture has previously been proven in certain cases, e.g. for dimension $\le 3$ (\cite{Broomhead, SVdBtoricII}), for quasi-symmetric GIT quotients for tori \cite{SVdB17} and for toric varieties associated to simplicial cones \cite{FMS19, BBB+}.
Many of these results rely on exhibiting tilting bundles on appropriate varieties or stacks, and indeed tilting theory is intrinsically linked to NCCRs. Iyama and Wemyss \cite{IW14b} prove (see Theorem \ref{Thm;Cor4.15IW}) that a tilting complex on a Gorenstein variety $Y$ mapping to $\spec R$ generates an NCCR if the map is projective, birational and crepant. In the case of affine Gorenstein toric varieties, we can assume that the variety is associated to a cone $\sigma=\cone(P\times \{1\})$ for some polytope $P$. Then, \v{S}penko and van den Bergh prove the existence of an NCCR if the Cox stack associated to a simplicialisation $\Sigma$ of $\sigma$ not introducing additional rays admits a tilting bundle. Our main result in this paper is a generalisation of both these results.

\begin{theorem*}[Theorem \ref{Prop:Parttiltingworks}]
    Choose a regular triangulation of $P$ and let $\Sigma$ be the corresponding fan refining $\sigma$. Let $\mathcal{T}$ be a partial silting complex on $\mathcal{X}_{\Sigma}$, the associated toric DM stack. Assume that $\Lambda=\End_{\mathcal{X}_{\Sigma}}(\mathcal{T})$ has finite global dimension. Then it is an NCCR for $R=k[\sigma^\vee\cap M]$.
\end{theorem*}

We note that a complex such as in the theorem is always a tilting complex (a combination of Lemma \ref{Lem:siltimpliestilt} and \cite[Proposition 2.5]{DHKR25}). For the remainder of the paper, we focus on the case  of $\sigma=\cone(P\times\{1\})$, where $\operatorname{Int}(P)$ contains a lattice point, examining some examples and describing a strategy to establish the existence of NCCRs in that case. For our considerations, we state an important corollary to the above theorem.

\begin{corollary*}[Corollary \ref{Cor:TiltCplxCanbdl}]
    Let $X$ be a complete simplicial toric variety such that $\mathcal{V}:=[\tot \omega_X]$ admits a tilting complex. Consider the cone $\sigma=|\Sigma_V|$, where $\Sigma_V$ is the fan associated to the toric vector bundle $\tot \omega_X$. Then $R=k[\sigma^\vee\cap M]$ has an NCCR.
\end{corollary*}

\noindent To find out when $\mathcal{V}$ admits an appropriate partial tilting complex of finite global dimension, we study and generalise results about tilting objects on global quotient stacks by Novakovi\'{c} \cite{Nova18}. 

In particular, we prove the following result. 
\begin{corollary*}[Corollary \ref{Cor:Nova+myresult}]
 Let $X_\Sigma$ be a projective simplicial toric variety and consider the cone $\sigma=|\Sigma_\mathcal{V}|$, where $\Sigma_\mathcal{V}$ is the fan associated to the toric vector bundle $\tot \omega_{X_\Sigma}$. Suppose $\mathcal{X}_{\Sigma}\cong [X/G]$ for some smooth, projective $X$ and a finite group $G$, that $\mathcal{T}$ is tilting on $[X/G]$ and that $[\tot \omega_{X_\Sigma}]=[\tot \mathcal{E}^\vee/G]$ for some $G$-equivariant vector bundle $\mathcal{E}$ on $X$. If $H^i(X,\mathcal{T}^\vee\otimes \mathcal{T}\otimes \operatorname{Sym}^l(\mathcal{E}))=0$ for $i\neq 0, l>0$, then $R=k[\sigma^\vee\cap M]$ has an NCCR.
\end{corollary*}

\noindent Since we cannot always assume that $\mathcal{X}_\Sigma$ is a global quotient stack of a smooth projective variety, we consider the case where $\Sigma$ is merely a simplicial fan and prove the following two results.

\begin{theorem*}[Theorem \ref{Thm:PullbackDMstack}]
Let $\Sigma$ be a simplicial fan and consider a vector bundle $V=\bigoplus_{j=1}^r\O_{X_\Sigma}(D_j)$, with a fan $\Sigma_V\subset (N\oplus \Z^r)_{\R}$ (constructible using \cite[\S 7]{CLS}). Then $X_{\Sigma_V}\cong\tot_{X_\Sigma}(V)$ and the lattice projection $N\oplus \Z^r\rightarrow N$ induces the bundle projection $p:X_{\Sigma_V}\rightarrow X_\Sigma$. Consider now the good moduli spaces $f_\Sigma:\mathcal{X}_\Sigma\rightarrow X_\Sigma$ and $f_{\Sigma_V}:\mathcal{X}_{\Sigma_V}\rightarrow X_{\Sigma_V}$. Then we have $\mathcal{X}_{\Sigma_V}\cong\tot_{\mathcal{X}_\Sigma}(f_\Sigma^\ast V)$ and a vector bundle projection $\pi:\mathcal{X}_{\Sigma_V}\rightarrow \mathcal{X}_\Sigma$. Suppose $\mathcal{X}_\Sigma$ admits a tilting complex $\mathcal{T}$. If $H^i(\mathcal{X}_\Sigma,\mathcal{T}^\vee\otimes\mathcal{T}\otimes \operatorname{Sym}^l(f_\Sigma^\ast V^\vee))=H^i(X_\Sigma,f_{\Sigma,\ast}(\mathcal{T}^\vee\otimes\mathcal{T})\otimes \operatorname{Sym}^l(V^\vee))=0$ for all $i\neq 0$ and $l\ge 1$, then $\pi^\ast\mathcal{T}$ is tilting on $\mathcal{X}_{\Sigma_V}$.
\end{theorem*}

\begin{corollary*}[Corollary \ref{Cor:NovaDMstack}]
In the situation of Theorem~\ref{Thm:PullbackDMstack}, if $V=\omega_{X_\Sigma}$, then set $N'=N\oplus \Z$ and $M'=M\oplus \Z$. Denote $\sigma=|\Sigma_V|\subseteq N'_{\R}$. Suppose $\sigma$ is a full-dimensional strongly convex rational polyhedral cone and that the induced toric morphism \[
q:X_{\Sigma_V}\rightarrow X_\sigma
\]
is projective (i.e. that the triangulation of $\sigma_{(1)}$ is regular). Then $R=k[\sigma^\vee\cap M']$ has an NCCR.
\end{corollary*}


\noindent The results we obtain are useful to examine cones over polytopes with interior points, as we illustrate on several examples. Finally, we focus on the case of reflexive Gorenstein cones and use results by Borisov and Hua \cite{BH09} to show that the following result holds.

\begin{theorem*}[Theorem \ref{Thm:FanoAlmostSimplicial}]
    Let $P\in N_{\R}\cong \R^n$ be a simplicial, reflexive polytope with $\le n+2$ vertices. Consider the cone $\sigma=\cone(P\times\{1\})$. Then $R=k[\sigma^\vee\cap M]$ has an NCCR.
\end{theorem*}

\subsection{Structure and Notation}

This paper is organised in three parts. We first discuss some background on toric geometry in \S\ref{sec:Toric}, revisiting the constructions of Cox stacks and toric VGIT. Then we discuss NCCRs in \S\ref{sec:NCCR}, proving the first main result of this paper in Theorem \ref{Prop:Parttiltingworks}. Finally, \S\ref{sec:CanBdls} discusses applications of this result to the case of Gorenstein cones obtained as support of the fan corresponding to the total space of canonical bundles over toric projective varieties. We establish some generalisations of known results regarding the existence of tilting complexes. In $\S$ \ref{sec:refl} we then focus our attention especially to the case of reflexive Gorenstein cones, which are intrinsically linked to Fano varieties.

Let us fix here some notation for the entirety of the paper. Given two toric fans $\Sigma, \Psi$ we will write $\Sigma\subset \Psi$ to mean that $\Sigma$ is a subfan of $\Psi$, i.e. $\sigma \in \Sigma\Rightarrow \sigma\in\Psi$. 
For a normal noetherian domain $R$, a finitely generated $R$-module is said to be \newterm{reflexive} if the canonical map $M\mapsto \Hom_R(\Hom_R(M,R),R)$ is an isomorphism. The category of reflexive $R$-modules is denoted by $\operatorname{ref}R$.  If $\Lambda$ is a reflexive $R$-algebra then $\operatorname{ref}(\Lambda)$ is the category of $\Lambda$-modules which are reflexive as $R$-modules. If $R$ is commutative as well, we say an $R$-module $M$ is \newterm{maximal Cohen-Macaulay} if $M_m$ is maximal Cohen-Macaulay for every maximal ideal $m$. The category of such modules is denoted by $\operatorname{CM}R$ and we will omit the word maximal in this paper. As detailed in \cite{IR08, IW14b}, if $R$ is further Gorenstein and $M$ is module-finite, $M\in \operatorname{CM}R\Leftrightarrow \Ext^i(M,R)=0$ for all $i>0$. In this paper we permit ourselves to write $\Rhom^i$ when not explicitly considering any equivariant structure, but will write $\Rhomi{i}_X^G(M,N) $ to mean the $i$-th $G$-equivariant derived functor on $X$, $\Rhom^G_X$.
Finally, we do fix here an algebraically closed field $k$ of characteristic 0.

\subsection{Acknowledgements}
The authors would like to thank Will Donovan, Angel Toledo, Tyler Kelly and David Favero for many discussions leading to the creation of this paper. We thank the anonymous referee for valuable comments and suggestions that significantly improved the paper. The first author is supported by Beijing Natural Science Foundation IS25013. The second author is supported by grants from Beijing Institute of Mathematical Sciences and Applications (BIMSA), the Beijing NSF BJNSF-IS24005, and the China National Science Foundation (NSFC) NSFC-RFIS program W2432008. He would like to also thank China's National Program of Overseas High Level Talent for generous support. 

\section{Toric stacks and VGIT}
\label{sec:Toric}

In this section, we will recall some facts from toric geometry, fixing the notation that we will use. Fix a lattice $M$ of rank $n$ with dual lattice $N$, equipped with the pairing $\langle -,-\rangle:M\times N\rightarrow Z$, extending $\R$-linearly to a pairing of $M_{R}=M\otimes_{\Z}\R, N_{\R}=N\otimes_{\Z}\R$.
\subsection{Cox stacks}
Let $\Sigma$ be a fan in $N_\R$. Using the Cox construction (see \cite{CLS}, Section 5.1), we can associate to $\Sigma$ a quotient stack $\mathcal{X}_\Sigma$. Denote by $\nu=\{ u_\rho\mid \rho\in\Sigma(1)\}\subset N$ the set of primitive lattice generators of the rays in $\Sigma(1)$. For the sake of simplifying the notation, we write $s:=|\nu|$. We can consider the vector space $\R^{s}$ with elementary $\Z$-basis vectors $e_\rho$, indexed by the rays $\rho\in\Sigma(1)$. Then we define the \newterm{Cox fan} of $\Sigma$ to be\[
\operatorname{Cox}(\Sigma):=\{\cone(e_\rho\mid\rho\in\sigma)\mid\sigma\in\Sigma\}.
\]
This is a subfan of the standard fan for $\A^{s}$, and thus its associated toric variety is an open subspace of said affine space. We denote this open set by $U_\Sigma:=X_{\operatorname{Cox}(\Sigma)}$, called the \newterm{Cox open set} associated to $\Sigma$. Consider now the right exact sequence

\begin{eqnarray}\label{eqn:Coxseqpre}
    && M\xrightarrow{f_\Sigma}\Z^s\xrightarrow{\pi}\operatorname{coker}(f_\Sigma)\rightarrow 0,\\
   &&  m\mapsto \sum_{\rho\in\Sigma(1)}\langle u_\rho,m\rangle e_\rho\nonumber.\;
\end{eqnarray}
Applying the functor $\Hom(-,\mathbb{G}_m)$ yields (as $\mathbb{G}_m$ is injective) the left exact sequence\[
1\rightarrow \Hom(\operatorname{coker}(f_\Sigma),\mathbb{G}_m)\xrightarrow{\hat{\pi}}\mathbb{G}_m^s\rightarrow \mathbb{G}_m^n,
\]

where $n=\dim M$. Note that the group $S_\Sigma:=\Hom(\operatorname{coker}(f_\Sigma),\mathbb{G}_m)$ acts on $U_\Sigma$.
\begin{definition}
    \label{Def:Coxstack}
    Define the \newterm{Cox stack} associated to $\Sigma$ to be\[
    \mathcal{X}_\Sigma:=[U_\Sigma/S_\Sigma].
    \]
\end{definition}

\noindent One compelling reason to consider such a quotient stack is that it may be smooth when the toric variety $X_\Sigma$ is not. In particular, we have the following result relating $\mathcal{X}_\Sigma$ and $X_\Sigma$ for simplicial fans.

\begin{theorem}[Theorem 4.12 in \cite{FK18}]
    \label{Thm:4.12FK18}
    If $\Sigma$ is simplicial, then $\mathcal{X}_\Sigma$ is a smooth Deligne-Mumford stack with coarse moduli space $X_\Sigma$. When $\Sigma$ is smooth (equivalently, the variety $X_\Sigma$ is smooth), $\mathcal{X}_\Sigma\cong X_\Sigma$.
\end{theorem}

Now that we have defined the Cox stacks that we associate to toric varieties, we can ask ourselves when these stacks may be in fact equivalent, if only at a derived level. We recall the following lemma. 

\begin{lemma}[Lemma 4.22 in \cite{FK18}]\label{Lem:4.22FK18}
    Suppose we have an exact sequence of algebraic groups,
    \[
    0\rightarrow H\xrightarrow{i}G\xrightarrow{\pi}Q\rightarrow 0.
    \]
    Let $G$ act on $X$ and hence on $X\times Q$ via $\pi$. Then we have an isomorphism of stacks\[
    [X\times Q/G]\cong [X/H].
    \]
\end{lemma}

\noindent Using this, we prove the following version of Corollary 4.23 in \cite{FK18}.

\begin{lemma}
    \label{Lem:StackIso}
Let $\Sigma$ be a simplicial fan with full-dimensional, strongly convex rational support. Suppose $\Psi$ is another such fan such that $\Sigma$ is a subfan of $\Psi$, $\Sigma\subset\Psi$. Then we have a stack isomorphism\[
[U_\Sigma\times\mathbb{G}_m^{|\Psi(1)\setminus\Sigma(1)|}/S_\Psi\times\mathbb{G}_m]\cong[U_\Sigma/S_\Sigma\times\mathbb{G}_m]
\]
\end{lemma}

\begin{proof}
    We note that by \eqref{eqn:Coxseqpre} we have right exact sequences
    \begin{eqnarray*}
        && M\xrightarrow{f_\Sigma}\Z^{|\Sigma(1)|}\rightarrow \operatorname{coker}(f_\Sigma)\rightarrow 0;\\
        &&  M\xrightarrow{f_\Psi}\Z^{|\Psi(1)|}\rightarrow \operatorname{coker}(f_\Psi)\rightarrow 0.\;
    \end{eqnarray*}
    Due to the full-dimensionality of $|\Sigma|, |\Psi|$, the maps $f_\Sigma, f_\Psi$ are injective and thus we can extend the sequences to the left by $0$. We can then build a commutative diagram (where we abuse notation and write $f_\Sigma, f_\Psi$ to denote the maps $(f_\Sigma,0), (f_\Psi,0)$):
    \[
    \begin{tikzcd}
        0\arrow[r]\arrow[d] &0\arrow[r]\arrow[d]& \ker b\arrow[r]\arrow[d] &\ker c\arrow[r]\arrow[d] & 0\arrow[d]\\
        0\arrow[r]\arrow[d] & M\arrow[d, equal, "a=id"]\arrow[r, "f_{\Psi}"] & \Z^{|\Psi(1)|}\times \Z \arrow[d, "b"]\arrow[r] & \operatorname{coker}(f_\Psi)\times \Z\arrow[d, "c"]\arrow[r] & 0 \arrow[d]\\
        0 \arrow[r]\arrow[d]& M\arrow[r, "f_{\Sigma}"] \arrow[d]& \Z^{|\Sigma(1)|}\times \Z\arrow[r]\arrow[d]& \operatorname{coker}(f_\Sigma)\times \Z\arrow[r] \arrow[d]& 0\arrow[d] \\
        0\arrow[r]& 0\arrow[r]& \operatorname{coker}b\arrow[r]&\operatorname{coker}c\arrow[r]& 0
    \end{tikzcd}
    \]
     
    We note here that $\ker a=\operatorname{coker}a=\operatorname{coker}b=0$ (as $a$ is the identity and $b$ the projection map). 
By the Snake lemma, there is an exact sequence\[
\ker a=0\rightarrow \ker b\rightarrow \ker c\rightarrow \operatorname{coker}a=0\rightarrow 0\rightarrow \operatorname{coker} c\rightarrow 0.
\]
Hence $\ker c=\ker b=\Z^{|\Psi(1)\setminus \Sigma(1)|}$. Thus, the diagram gives the following exact sequence (coming from the right vertical sequence in the diagram):
\[
0\rightarrow \Z^{|\Psi(1)\setminus\Sigma(1)|}\rightarrow \operatorname{coker}(f_\Psi)\times \Z\rightarrow \operatorname{coker}(f_\Sigma)\times \Z\rightarrow 0.
\]
Applying $\Hom(-,\mathbb{G}_m)$ to this sequence yields
\[
0\rightarrow S_\Sigma\times \mathbb{G}_m\rightarrow S_\Psi\times \mathbb{G}_m\rightarrow \mathbb{G}_m^{|\Psi(1)\setminus\Sigma(1)|}\rightarrow0.
\]
The desired stack isomorphism $[U_\Sigma\times\mathbb{G}_m^{|\Psi(1)\setminus\Sigma(1)|}/S_\Psi\times\mathbb{G}_m]\cong[U_\Sigma/S_\Sigma\times\mathbb{G}_m]$ follows by Lemma \ref{Lem:4.22FK18}.
\end{proof}

Our focus in this paper lies on toric varieties whose fans have convex rational polyhedral support enjoying the property of being Gorenstein cones.
\begin{definition}
    \label{Def:QGor}
    A cone $\sigma\subset N_\R$ is said to be \newterm{Gorenstein} (resp. $\Q$-Gorenstein) \newterm{with respect to} $\mathfrak{m}_\sigma\in M$ (resp. $\mathfrak{m}_\sigma\in M_\Q$) if the cone is generated over $\Q$ by finitely many lattice points in $\{n\in N\mid \langle \mathfrak{m}_\sigma,n\rangle=1\}$.
\end{definition} 

The element $\mathfrak{m}_\sigma$ dictating the Gorenstein structure of $\sigma$ will be referred to as \newterm{Gorenstein element} by us. Note that if $\sigma$ is $\Q$-Gorenstein and full-dimensional, i.e. $\dim \sigma=n$, then $\mathfrak{m}_\sigma$ is unique. Changing basis, we may assume $\mathfrak{m}_\sigma=(0,\dots,0,1)$ and hence that $\sigma$ is of the form $\cone(P\times\{1\})$ for some convex lattice polyhedron $P\subset \R^{n-1}\times\{1\}\subset \R^n=N_{\R}$. 

Fixing such a Gorenstein cone $\sigma=\cone(P\times\{1\})$, consider a simplicial fan $\Sigma\subseteq N_{\R}$ such that $|\Sigma|=\sigma$ and such that the primitive generators $u_\rho$ of $\rho\in\Sigma(1)$ lie in $P\times\{1\}$. In other words, we require $\langle(0,\dots,0,1),u_\rho\rangle=1$ for all $\rho\in\Sigma(1)$. Denote by $R$ the ring $R=k[|\Sigma|^\vee\cap M]$. In their paper \cite{SVdBtoricII}, \v{S}penko and Van den Bergh prove the following lemma.

\begin{lemma}[Lemma A.3 in \cite{SVdBtoricII}]
\label{Lem:RelSerFun}
    The relative Serre functor of $\dbcoh{\mathcal{X}_\Sigma}$ with respect to $\dbcoh{X_\sigma}$ is the identity; i.e. for all $\mathcal{F}, \mathcal{G}\in\dbcoh{\mathcal{X}_\Sigma}$ we have\[
    \Rhom_{\mathcal{X}_\Sigma}(\mathcal{F},\mathcal{G})\cong\Rhom_R(\Rhom_{\mathcal{X}_\Sigma}(\mathcal{G},\mathcal{F}),R).
    \] 
\end{lemma}

\subsection{Factorisation categories and VGIT}

In this paper, we present the Cox stacks as GIT quotients and we study and compare such stacks by varying the geometric invariant theory. This process, which we abbreviate as VGIT, is aided by expressing the derived categories in terms of \newterm{factorisation categories}, introduced here below. Good resources for more in-depth treatments of these categories include \cite{Hirano, BFK14}.

Let $k$ be an algebraically closed field of characteristic zero and let $X$ be a smooth variety over $k$ with an action of an affine algebraic group $g$ on it. Consider a $G$-invariant section $W$ of an invertible $G$-equivariant sheaf $\mathcal{L}$, i.e. $W\in\Gamma(X,\mathcal{L})^G$. The data $(X,G,W)$ is a \newterm{gauged Landau-Ginzburg (LG) model}. To a gauged LG model, we will now construct an associated \newterm{absolute derived category}.

\begin{definition}
    \label{Def:Fact}
    Let $\E_0, \E_1$ be two $G$-equivariant quasi-coherent sheaves and together with two morphisms $\phi^\E_0, \phi^\E_1$ fitting into the following sequence\[
    \E_1\xrightarrow{\phi_0^\E}\E_0\xrightarrow{\phi_1^\E}\E_1\otimes_{\O_X}\mathcal{L},
    \] 
    such that $\phi_1^\E\circ\phi_0^\E=W=(\phi_0^\E\otimes_{\O_X}\mathcal{L})\circ\phi_1^\E$. The data $(\E_0,\E_1,\phi^\E_0,\phi_1^\E)$ is called a \newterm{factorisation}.
\end{definition}

Given two factorisations $\E, \mathcal{F}$, we can define a complex of morphisms between them in the following way. Consider the graded vector space\[
\Hom(\E, \mathcal{F})^\bullet := \bigoplus_{n\in \Z} \Hom(\E, \mathcal{F})^n,
\]
where
\begin{equation*}\begin{aligned}
\Hom(\E, \mathcal{F})^{2m} &:= \Hom(\E_1, \mathcal{F}_1\otimes \mathcal{L}^{\otimes m}) \oplus \Hom(\E_0, \mathcal{F}_0\otimes \mathcal{L}^{\otimes m}), \\
\Hom(\E, \mathcal{F})^{2m+1} &:= \Hom(\E_1, \mathcal{F}_0\otimes \mathcal{L}^{\otimes m}) \oplus \Hom(\E_0, \mathcal{F}_1\otimes \mathcal{L}^{\otimes m+1}). \\
\end{aligned}\end{equation*}
In particular we note that a morphism $f:\mathcal{E}\rightarrow\mathcal{F}[n]$ is a pair $(f_0,f_1)$. 
The differential $d^i:  \Hom(\E, \mathcal{F})^i \to  \Hom(\E, \mathcal{F})^{i+1}$ is given by $d^i(f) = \phi^{\mathcal{F}}_{\star + i} \circ f - (-1)^i f \circ \phi^{\E}_{\star}$.
In this way, we obtain a dg category $\Fact{X,G,W}$ with objects factorisations and morphisms defined above. 
\begin{definition}
    Let $\fact{X,G,W}$ be the full dg-subcategory of $\Fact{X,G,W}$ whose components are coherent.
\end{definition}

 Let $Z^0\Fact{X,G,W}$ be the subcategory of $\Fact{X,G,W}$ with the same objects, but where we only allow for degree zero morphisms. Given a complex of objects in $Z^0\Fact{X,G,W}$, one can construct a new object $\mathcal{T}\in\Fact{X,G,W}$, called the \newterm{totalisation} of the complex.
Write the complex as \[
...\rightarrow \E^i\xrightarrow{f^i}\E^{i+1}\xrightarrow{f^{i+1}}\dots.
\]
The totalisation $\mathcal{T}\in\Fact{X,G,W}$ is given by the data\begin{eqnarray*}
    && T_0:= \bigoplus_{i=2l}\E_{0}^i\otimes_{\O_X}\mathcal{L}^{-l}\oplus\bigoplus_{i=2l+1}\E_1^i\otimes_{\O_X}\mathcal{L}^{-l};\\
    && T_1:= \bigoplus_{i=2l}\E^i_1\otimes_{\O_X}\mathcal{L}^{-l}\oplus\bigoplus_{i=2l-1}\E_0^i\otimes_{\O_X}\mathcal{L}^{-l};\\
    && \phi_0^{\mathcal{T}}:= \bigoplus_{i=2k}f_0^i\otimes\mathcal{L}^{-k}\oplus\bigoplus_{i=2k-1}f_1^i\otimes\mathcal{L}^{-k};\\
    && \phi_1^\mathcal{T}:= \bigoplus_{i=2k}f_1^i\otimes\mathcal{L}^{-k}\oplus\bigoplus_{i=2k-1}f_0^i\otimes\mathcal{L}^{-k}.\;
\end{eqnarray*}

Define $\operatorname{Acyc}(X,G,W)$ to be the full subcategory of $\Fact{X,G,W}$ consisting of all totalisations of bounded exact complexes in $Z^0\Fact{X,G,W}$, and let $\operatorname{acyc}(X,G,W)=\operatorname{Acyc}(X,G,W)\cap \fact{X,G,W}$.

\begin{definition}
    \label{Def:Dabs}
    The \newterm{absolute derived category} $\dabs{X,G,W}$ is the idempotent completion of the Verdier quotient of $\fact{X,G,W}$ by $\operatorname{acyc}(X,G,W)$.
\end{definition}

\noindent This definition is analogous to how the usual derived category of an abelian can be defined as idempotent completion of the Verdier quotient of the homotopy category by acyclic complexes. The absolute derived category $\dabs{X,G,W}$ can be thought of as the derived category of the gauged LG model $(X,G,W)$. To justify this claim, we must introduce some context and notation. 

\begin{notation}\label{geometric context notation}
Let $Y$ be a smooth quasi-projective variety with a $G$-action. Suppose that $s$ is a regular section of a $G$-equivariant vector bundle $\E$ on $Y$ with vanishing locus $Z := Z(s)$. Let $\mathbb{G}_m$ act on the total space $\tot \E^\vee$ of the dual bundle to $\E$  by fiberwise dilation (the so-called \newterm{$R$-charge}) and consider the pairing $W=\langle -,s\rangle$ as a section of $\O_{\tot \E^\vee}(\chi)$ where $\chi$ is the projection character. 
\end{notation}

We have the following theorem, which has appeared in various forms due to Orlov\cite{Orlov92}, Isik\cite{Isik}, Shipman\cite{Shipman}, and, in the form we use here, Hirano \cite{Hirano}.

\begin{theorem}[Proposition 4.8 of \cite{Hirano}]\label{Orlovs thm}
There exists an equivalence of categories
$$
\Omega: \dbcoh{[Z/G]} \stackrel{\sim}{\longrightarrow} \dabs{\tot \E^\vee, G \times \mathbb{G}_m, W}.
$$
\end{theorem}

 A consequence of Theorem \ref{Orlovs thm} is the following.
 \begin{corollary}[Corollary 2.3.12 in \cite{BFK19}]
     \label{Cor:0superpotential}
     There is an equivalence of categories \[
     \dbcoh{[X/G]}\simeq \dabs{X,G\times \mathbb{G}_m,0}.
     \]
 \end{corollary}

Using the notions of absolute derived categories, together with the above Corollary \ref{Cor:0superpotential}, we can formulate a way to compare derived categories related by VGIT.

Let $\sigma \subseteq N_{\R}$ be a $\Q$-Gorenstein cone and $\nu \subseteq \sigma \cap N$ be a finite, geometric collection of lattice points which contains the (primitive) ray generators of $\sigma$. Partition the set $\nu$ into two subsets
\begin{equation}\begin{aligned}
\nu_{=1} = \{v \in \nu \ | \ \langle \mathfrak{m}_{\sigma}, v\rangle = 1\};\\
\nu_{\ne 1} = \{ v \in \nu \ | \ \langle \mathfrak{m}_{\sigma}, v \rangle \ne 1\}. 
\end{aligned}\end{equation}
Note that since $\sigma$ is $\Q$-Gorenstein, the ray generators of $\sigma$ are contained in $\nu_{=1}$. Consider two simplicial fans $\Sigma, \Psi$ with support $\cone(\nu)$. Fix a $S_{\nu}$-invariant\footnote{Recall the group $S_\nu$ as in the discussion preceding Definition \ref{Def:Coxstack}.} function $\bar{W}$ which is a global function on both of the affines $U_\Sigma, U_\Psi$.
\begin{theorem}[Theorem 5.8 of \cite{FK18}]\label{Thm:5.8inFK18}
Let $\Psi$ be any simplicial fan such that $\Psi(1) = \{ \cone{(v)} \ | \ v \in \nu\}$ and $X_{\Psi}$ is semiprojective. Similarly, let $\tilde \Sigma$ be any simplicial fan such that $\tilde\Sigma(1) \subseteq \nu_{=1}$, $X_{\tilde \Sigma}$ is semiprojective and $\cone{(\tilde{\Sigma}(1))} = |\Psi|$. We have the following:
\begin{enumerate}
\item If $\langle m_{\sigma}, a\rangle > 1$ for all $a \in \nu_{\ne 1}$, then there is a fully-faithful functor
$$
\dabs{U_{\tilde \Sigma} \times \mathbb{G}_m^{\nu \setminus \tilde\Sigma(1)}, S_{\Psi} \times \mathbb{G}_m, \bar W} \to \dabs{U_{\Psi}, S_{\Psi} \times \mathbb{G}_m, \bar W}.
$$
\item If $\langle m_{\sigma}, a\rangle < 1$ for all $a \in \nu_{\ne 1}$, then there is a fully-faithful functor
$$
\dabs{U_{\Psi}, S_{\Psi} \times \mathbb{G}_m, \bar W} \to \dabs{U_{\tilde \Sigma} \times \mathbb{G}_m^{\nu \setminus \tilde\Sigma(1)}, S_{\Psi} \times \mathbb{G}_m, \bar W}.
$$
\item If $\nu_{\ne 1} = \varnothing$, then there is an equivalence
$$
\dabs{U_{\tilde \Sigma} \times \mathbb{G}_m^{\nu \setminus \tilde\Sigma(1)}, S_{\Psi} \times \mathbb{G}_m, \bar W} \cong \dabs{U_{\Psi}, S_{\Psi} \times \mathbb{G}_m, \bar W}.
$$
\end{enumerate}
\end{theorem}

Combining the Theorem \ref{Thm:5.8inFK18} and Corollary \ref{Cor:0superpotential}, we obtain the following result.

\begin{corollary}
    \label{Cor:VGITforGor}
    Let $\Psi, \Sigma$ be two simplicial fans such that $u_\rho\in\nu_{=1}$ for $\rho\in\Sigma(1)\cup\Psi(1)$ (where $u_\rho$ denotes the primitive generators of the ray $\rho$), $|\Psi|=|\Sigma|=\cone(\conv(\nu_{=1}))$ and $X_{\Psi}, X_{\Sigma}$ semiprojective. Then\[
    \dbcoh{\mathcal{X}_\Psi}\cong\dbcoh{\mathcal{X}_{\Sigma}}
    \]
\end{corollary}

\begin{proof}
    Consider the set $\nu'=\nu_{=1}$. Then there exists a simplicial fan $\Phi$ such that $X_{\Phi}$ is semiprojective and $\{u_\rho\mid \rho\in\Phi(1)\}=\nu'$ (by abuse of notation we confound the primitive generators with lattice points of the same coordinates). 
    We note that $\nu'_{\ne 1}=\emptyset$, and so we can apply Theorem \ref{Thm:5.8inFK18} to obtain
    \[
    \dabs{U_\Sigma\times \mathbb{G}_m^{|\nu' \setminus \Sigma(1)|}, S_{\Phi} \times \mathbb{G}_m, 0} \cong \dabs{U_{\Phi}, S_{\Phi} \times \mathbb{G}_m, 0}\cong\dabs{U_\Psi\times \mathbb{G}_m^{|\nu' \setminus \Psi(1)|}, S_{\Phi} \times \mathbb{G}_m, 0}. 
    \]
    We note that by Corollary \ref{Cor:0superpotential}, $\dabs{U_\Sigma, S_\Sigma\times \mathbb{G}_m,0}\cong \dbcoh{\mathcal{X}_\Sigma}$, and analogously for $\Psi$. Thus, it suffices to show
    $ \dabs{U_\Sigma\times \mathbb{G}_m^{|\nu' \setminus \Sigma(1)|}, S_{\Phi} \times \mathbb{G}_m, 0}\cong \dabs{U_\Sigma, S_\Sigma\times \mathbb{G}_m,0}$, together with the corresponding equivalence for $\Psi$. These equivalences follow from the stack isomorphism $[U_\Sigma\times \mathbb{G}_m^{|\nu'\setminus\Sigma(1)|}/S_{\nu'}\times\mathbb{G}_m]\cong[U_\Sigma/S_\Sigma\times\mathbb{G}_m]$ (Lemma \ref{Lem:StackIso}).
\end{proof}

\section{Non-commutative crepant resolutions}
\label{sec:NCCR}

When encountering singular objects, as algebraic geometers we have two (often intertwined) options. Either we study the singularities for their own sake, or we attempt resolving them. In this section, we introduce a notion towards the second option, introduced by Van den Bergh. More in-depth treatments on the matter can be found in \cite{VdB04, VdB23, Leu12}.

\noindent Underlying the notion of non-commutative resolutions is the interplay of algebraic geometry and non-commutative rings, which can for instance be observed in the work of Beilinson \cite{Bei78}. An important concept in this context is that of a \newterm{tilting complex}.

\begin{definition}
    \label{Def:tilt}
     Let $Y$ be a Noetherian scheme. A \newterm{partial tilting complex} $\mathcal{T}$ on $Y$ is a perfect complex such that $\Ext^i_Y(\mathcal{T},\mathcal{T})=0$ for $i\neq 0$. A \newterm{tilting complex} is a partial tilting complex that \newterm{generates} $D_{Qch}(Y)$ in the sense that its right orthogonal is zero, i.e.
    $\Rhom_Y(\mathcal{T},\mathcal{F})=0$ implies $\mathcal{F}=0$. A \newterm{(partial) tilting bundle} is a (partial) tilting complex which is a vector bundle.
\end{definition}
The notion of tilting complexes and bundles has a natural generalisation to algebraic stacks. For an algebraic stack $\mathcal{Y}$, we define a perfect complex $\mathcal{T}$ to be partial tilting if $\Ext_{\mathcal{Y}}^i(\mathcal{T},\mathcal{T})=0$ for $i\neq 0$. If $\Ext_{\mathcal{Y}}^i(\mathcal{T},\mathcal{T})$ only vanishes for $i>0$, we instead refer to the complex as partial silting. The perfect complex $\mathcal{T}$ is tilting (respectively silting) if furthermore it generates $\operatorname{D}_{Qch}(\mathcal{Y})$, which for smooth, separated noetherian DM stacks equates the usual generating property for $\dbcoh{\mathcal{Y}}$.

The following (simplified version of a) powerful result displays the links between the algebraic geometry and non-commutative ring theory.
\begin{theorem}[Theorem 1.2 in \cite{KK20}, 1.7 in \cite{VdB23}]
    \label{Thm:KK20Thm1.2}
    If $\mathcal{T}$ is a tilting complex on a noetherian scheme $Y$ then $\Rhom_Y(\mathcal{T},-)$ defines an equivalence of categories between $\operatorname{D}_{Qch}(Y)$ and $\operatorname{D}(\Lambda^\circ)$ for $\Lambda=\End_Y(\mathcal{T})$. Moreover, if $Y$ is regular then $\Lambda$ has finite global dimension. If furthermore, $\Lambda$ is right noetherian then $\Rhom_Y(\mathcal{T},-)$ restricts to an equivalence of categories\[
    \dbcoh{Y}\simeq \dbmod\Lambda
    \]
\end{theorem}

From hereon, let $R$ be a normal noetherian domain with quotient field $K$. A reflexive Azumaya $R$-algebra in codimension one is called a \newterm{reflexive Azumaya algebra}. Recall here that an Azumaya algebra is an algebra such that $A^{op}\otimes_R A$ is Morita equivalent to $R$. We say that a reflexive Azumaya algebra $\Lambda$ is trivial if it is of the form $\End_R(M)$ for some reflexive $R$-module $M$. Then $\operatorname{ref}(R)$ and $\operatorname{ref}(\Lambda)$ are equivalent.

\begin{definition}
    \label{Def:NCCR}
    A \newterm{non-commutative resolution} of $R$ is a trivial reflexive Azumaya algebra $\Lambda=\End_R(M)$ for $M\in \operatorname{ref}R$ such that $\gldim\Lambda<\infty$. Assuming further that $R$ is Gorenstein, we call $\Lambda$ \newterm{crepant} if it is additionally a Cohen-Macaulay $R$-module.
\end{definition}
 We will often abbreviate the term non-commutative crepant resolution with \newterm{NCCR}.

 \begin{remark} This definition captures the concept of crepant resolutions as we think of them geometrically in a non-commutative manner. In particular, the properties characterising resolutions that we aim to substitute are regularity and crepancy. 
Regularity corresponds to the condition that the global dimension is finite, and being crepant corresponds to being Cohen-Macaulay. Details of this can be found in $\S4$ in \cite{VdB04}.

The definitions above can also be understood for (potentially non-affine) schemes, by reducing them to the affine patches.
\end{remark}

As detailed in \cite{VdB23}, since the introduction of the concept by Van den Bergh, NCCRs have been constructed in a variety of cases. Notable examples are the construction of NCCRs for non-commutative resolutions of quotient singularities for reductive groups \cite{SVdB17}, for affine toric simplicial cones \cite{FMS19}, and for $\le 3$ dimensional affine Gorenstein toric varieties. The methodology is varied, ranging from using (partial) tilting complexes on crepant resolutions, to representation theory and mutation theory (see e.g. \cite{SVdB17, IW14a}) and the use of combinatorial objects like dimer models in toric cases \cite{Broomhead}.

In this paper, we will focus on the case of affine Gorenstein toric varieties, aiming to make progress towards answering the following Conjecture.
\begin{conjecture}
    \label{Conj:affinetoric}
    An affine Gorenstein toric variety always has an NCCR.
\end{conjecture}

Consider a full-dimensional strongly convex rational cone $\sigma\subset N_{\R}$. This defines an algebra $R_\sigma=k[\sigma^\vee\cap M]$, called a \newterm{toric} algebra, and every affine toric variety arises this way as $X=\spec R_\sigma$ for some cone $\sigma$. Such an algebra is called \newterm{Gorenstein} if the cone $\sigma$ is a Gorenstein cone.

An NCCR is said to be \newterm{toric} if the reflexive module defining the NCCR is isomorphic to a sum of ideals. Broomhead answered the 3-dimensional case of Conjecture \ref{Conj:affinetoric} using dimer models, which has later been reproven by \v{S}penko and Van den Bergh by constructing tilting bundles on refinements of the underlying Gorenstein cone.
\begin{theorem}[Theorem 8.6 \cite{Broomhead}]
    \label{Thm:Broomhead}
    The coordinate ring of a 3-dimensional Gorenstein affine toric variety admits a toric NCCR.
\end{theorem}

For their alternative proof of this result, \v{S}penko and Van den Bergh rely on the following Proposition, which we generalise in the form of Theorem \ref{Prop:Parttiltingworks}.
\begin{proposition}[Proposition 3.3 in \cite{SVdBtoricII}]
    \label{Prop:3.3SVDBtoricII}
    Given a Gorenstein cone $\sigma$, write it in the form $\cone(P\times\{1\})$ for $P$ a lattice polytope. Choose a \textbf{regular} triangulation of $P$ without extra vertices and let $\Sigma$ be the corresponding fan. Let $\mathcal{T}$ be a tilting bundle on $\mathcal{X}_{\Sigma}$, the associated DM stack. Then $\Lambda=\End_{\mathcal{X}_{\Sigma}}(\mathcal{T})$ is an NCCR for $R=k[\sigma^\vee\cap M]$ corresponding to $M'=\Gamma(\mathcal{X}_{\Sigma},\mathcal{T})$.
\end{proposition}
\begin{remark}
We note here that we added the word \textbf{regular} to the statement as it would be found in \cite{SVdBtoricII}. 
This was also noted in the formulation of Lemma 3.5 in \cite{VdB23}, a paper which gives a wonderful summary of the theory surrounding NCCRs.
Let us here remark that in the case of two-dimensional polytopes $P$, which is ultimately the application of Proposition \ref{Prop:3.3SVDBtoricII} in \cite{SVdBtoricII}, all triangulations of $P$ not using any interior vertices are, in fact, regular. This can be proven by induction on the number of vertices (with the case of exactly $\dim P+1$ vertices being trivial), noting that any triangulation by the two-ears theorem contains an "exterior'' triangle, i.e. a vertex with exactly two outgoing edges (connected to the adjacent vertices). Applying the induction hypothesis and Lemma 4.5 in \cite{KM24} shows that the chosen triangulation of $P$ was indeed a regular triangulation. 

In dimensions higher than 2, this fails. For example, in dimension 3 there is a convex polytope known as "capped triangular prism", which allows for 2 non-regular triangulations not using interior vertices (see $\S$ 16.3.1 in  \cite{HandbookDiscrete}).

\end{remark}

The proof of Theorem \ref{Thm:Broomhead} presented in \cite{SVdBtoricII} now reduces to showing tilting bundles as in Proposition \ref{Prop:3.3SVDBtoricII} always exist by embedding $P$ into a bigger polytope which has known tilting bundles (using work of Gulotta \cite{Gul} and Ishii-Ueda \cite{IU08, IU15, IU16}) and restricting these bundles. The restriction is non-trivial as tilting is not a local property and relies on the fact that they consider the three-dimensional cases.

A second result of interest generating NCCRs from tilting objects is due to Iyama and Wemyss \cite{IW14b}. 
\begin{theorem}[Corollary 4.15 in \cite{IW14b}]
\label{Thm;Cor4.15IW}
    Let $f:Y\rightarrow \spec R$ be a projective birational morphism between $d$-dimensional Gorenstein varieties. Suppose that $Y$ is derived equivalent to some ring $\Lambda$, then the following are equivalent.
    \begin{enumerate}
        \item $f$ is a crepant resolution of $\spec R$.
        \item $\Lambda$ is an NCCR of $R$.
    \end{enumerate}
\end{theorem}
In particular, for $Y$ a Gorenstein variety in the above setting and $\mathcal{T}$ a tilting complex on $\dbcoh{Y}$, $\End(\mathcal{T})$ is an NCCR of $R$. 
Two ingredients of the proof that we will be using ourselves are the following results.

\begin{lemma}[Lemma 4.3 in \cite{IW14b}]\label{Lem:4.3IW14}
Let $Y\rightarrow \spec R$ be a projective birational morphism between $d$-dimensional normal integral schemes. Let $\mathfrak{p}\in\spec R$ and consider the following pullback diagram:
\[\begin{tikzcd}
    Y'\arrow[r, "i"]\arrow[d, "g"]& Y\arrow[d, "f"]\\
    \spec R_\mathfrak{p}\arrow[r, "j"] & \spec R
\end{tikzcd}\]
\begin{enumerate}
    \item If $\mathfrak{p}$ is a height one prime, then $g$ is an isomorphism.
    \item If $\mathcal{V}$ is a partial tilting complex of $Y$ with $\Lambda:=\End_{\dbcoh{Y}}(\mathcal{V})$, then $\Lambda$ is a module-finite $R$-algebra and $i^\ast\mathcal{V}$ is a partial tilting complex of $Y'$ with $\Lambda_\mathfrak{p}\cong\End_{Y'}(i^\ast\mathcal{V})$
    \end{enumerate}
\end{lemma}

\begin{proposition}[Auslander-Goldman, found as Proposition 4.4 in \cite{IW14b}]
    \label{Prop:Prop4.4IW}
    Let $R$ be a normal domain and let $\Lambda$ be a module-finite $R$-algebra. Then the following conditions are equivalent.
    \begin{enumerate}
        \item There exists $M\in\operatorname{ref}R$ such that $\Lambda\cong\End_R(M)$ as $R$-algebras.
        \item $\Lambda\in\operatorname{ref}R$ and further $\Lambda_\mathfrak{p}$ is Morita equivalent to $R_\mathfrak{p}$ for all $\mathfrak{p}\in\spec R$ with $\operatorname{ht}\mathfrak{p}=1$.
    \end{enumerate}
\end{proposition}

\subsection{NCCRs via partial silting complexes}

We are now ready to state and prove the generalisation of Proposition \ref{Prop:3.3SVDBtoricII} and Theorem \ref{Thm;Cor4.15IW} that lies at the core of this paper. Recall that we can, and henceforth will, write a Gorenstein cone (after suitable change of basis) as $\sigma=\cone(P\times\{1\})\subseteq N_{\R}\times \R$, where $P$ is a lattice polytope.

\begin{lemma}
    \label{Lem:siltimpliestilt}
    Choose a regular triangulation of $P$ and let $\Sigma$ be the corresponding fan refining $\sigma$. Let $\mathcal{T}$ be a partial silting complex on $\mathcal{X}_{\Sigma}$, the associated toric DM stack. Then $\mathcal{T}$ is in fact partial tilting on $\mathcal{X}_\Sigma$.
    If further $\Lambda=\End_{\mathcal{X}_\Sigma}(\mathcal{T})$ has finite global dimension, then $\mathcal{T}$ is a tilting complex.
\end{lemma}
\begin{proof}
Denote by $R$ the coordinate ring of the toric variety associated to $\sigma$, i.e. $R=k[\sigma^\vee\cap M]$. 
Consider the endomorphism algebra $\Lambda=\Rhom_{\mathcal{X}_\Sigma}(\mathcal{T},\mathcal{T})$. The relative Serre functor from Lemma \ref{Lem:RelSerFun} implies\[
\Lambda\cong \Rhom_R(\Lambda,R).
\]
The partial silting condition means that $H^i(\Lambda)=0$ for $i>0$. Thus, $\Lambda\in \operatorname{D}^{\le 0}(\operatorname{mod}R)$, where $(\operatorname{D}^{\le0}(R),D^{\ge0}(R))$ is the standard $t$-structure on $\dbmod{R}$. Suppose now $j<0$. Then \[
H^j\Rhom_R(\Lambda,R)=\Hom_{\dbmod{R}}(\Lambda,R[j]).
\]
Now, $\Lambda\in \operatorname{D}^{\le0}(R)$ and $j<0$ and thus $R[j]\in\operatorname{D}^{\ge1}(R)$. Hence, $\Rhom_R(\Lambda,R)\in\operatorname{D}^{\ge0}(R)$. Since $\operatorname{D}^{\le0}(R)\cap\operatorname{D}^{\ge0}(R)=\operatorname{D}^{0}(R)$, we have that $\Rhom_R(\Lambda,R)=\Lambda$ is concentrated in degree 0 -  which means that it is actually partial tilting.  

The second part of the statement, which is that a partial tilting complex with finite global-dimensional endomorphism algebra is in fact a tilting complex, follows by the same proof as for Proposition 2.5 in \cite{DHKR25}.
\end{proof}
\begin{theorem}
    \label{Prop:Parttiltingworks}
    Choose a regular triangulation of $P$ and let $\Sigma$ be the corresponding fan refining $\sigma$. Let $\mathcal{T}$ be a partial silting complex on $\mathcal{X}_{\Sigma}$, the associated toric DM stack. Assume that $\Lambda=\End_{\mathcal{X}_{\Sigma}}(\mathcal{T})$ has finite global dimension. Then it is an NCCR for $R=k[\sigma^\vee\cap M]$.
\end{theorem}

\begin{proof}

    We first note that $\Sigma$ is simplicial, and thus the DM stack $\mathcal{X}_\Sigma$ is smooth (by Theorem \ref{Thm:4.12FK18}). By assumption, $\Lambda$ has finite global dimension. To prove that $\Lambda$ is an NCCR, we want to show that $\Lambda=\End_R(M)$ for $M\in \operatorname{ref}R$ and that $\Lambda$ is maximal Cohen-Macaulay. We will essentially, additionally using some categorical implications of VGIT of toric DM stacks, be imitating the proof of Theorem 4.5 in \cite{IW14b} to obtain module-finiteness of $\Lambda$. Then the relative Serre duality will give us maximal Cohen-Macaulayness of $\Lambda$, which allows us to finish imitating the proof of Iyama-Wemyss and show that $\Lambda$ has the required form.
    
    By Theorem 11.1.9 in \cite{CLS}, there is a refinement of $\sigma$ corresponding to a fan $\Sigma_\Delta$ such that $\Sigma_\Delta$ is smooth and the toric morphism $f:X_{\Sigma_\Delta}\rightarrow X_\sigma$ is a projective resolution of singularities. The refinement $\Sigma_\Delta$ is obtained by a sequence of star subdivisions on points in $N\cap\{\sum_{\rho\in\sigma(1)}\lambda_i u_\rho\mid 0\le \lambda_i<1\}$. 
    
    Denote by $\nu$ the set of primitive generators of $\Sigma_\Delta(1)$. Then $\langle\mathfrak{m}_\sigma,v\rangle\geq 1$ for all $v\in\nu$, where $\mathfrak{m}_\sigma\in M$ is the Gorenstein element of $\sigma$ (i.e. $\langle \mathfrak{m}_\sigma, u_\rho\rangle=1\quad \forall\rho\in\sigma(1)$). We are now in a position to apply Theorem \ref{Thm:5.8inFK18} (and implicitly Lemma \ref{Lem:StackIso} and Theorem \ref{Orlovs thm}) to obtain a fully faithful functor
    \[
    F:\dbcoh{\mathcal{X}_\Sigma}\rightarrow \dbcoh{\mathcal{X}_{\Sigma_\Delta}}.
    \]
    Since the functor is fully faithful, $\mathcal{V}=F\mathcal{T}$ is partial tilting on $\mathcal{X}_{\Sigma_\Delta}$ and $\End_{\mathcal{X}_{\Sigma_\Delta}}(\mathcal{V})=\End_{\mathcal{X}_{\Sigma}}(\mathcal{T})=\Lambda$. As this equivalence is $R$-linear, we can identify these endomorphism algebras as $R$-algebras. By Theorem \ref{Thm:4.12FK18}, $\mathcal{X}_{\Sigma_\Delta}\cong X_{\Sigma_\Delta}$ and so we can apply Lemma \ref{Lem:4.3IW14}. So, given $\mathfrak{p}\in\spec R$, consider the pullback diagram:
    \[
    \begin{tikzcd}
    X_{\Sigma_\Delta}'\arrow[r, "i"]\arrow[d, "g"]& X_{\Sigma_\Delta}\arrow[d, "f"]\\
    \spec R_{\mathfrak{p}}\arrow[r, "j"] & \spec R.
    \end{tikzcd}
    \]
    Hence $\Lambda$ is a module-finite $R$-algebra and $i^\ast\mathcal{V}$ is partial tilting on $X_{\Sigma_\Delta}'$ with $\Lambda_\mathfrak{p}\cong \End_{X_{\Sigma_\Delta}'}(i^\ast\mathcal{V})$ and if $\mathfrak{p}$ is a height one prime, then $g$ is an isomorphism.
    Apply now the relative Serre functor given by Lemma \ref{Lem:RelSerFun} to obtain
    \[
   \Rhom_{\mathcal{X}_\Sigma}(\mathcal{T},\mathcal{T})\cong\Rhom_{R}(\Rhom_{\mathcal{X}_\Sigma}(\mathcal{T},\mathcal{T}),R).
    \]
    As $\mathcal{T}$ is partial tilting, $\Rhom_{\mathcal{X}_\Sigma}(\mathcal{T},\mathcal{T})=\End_{\mathcal{X}_\Sigma}(\mathcal{T})$, so we get\[
    \End_{\mathcal{X}_\Sigma}(\mathcal{T})=\Rhom_R(\End_{\mathcal{X}_\Sigma}(\mathcal{T}),R).
    \]
    To both sides, we apply $H^i$, giving
    \[
    0=\Ext^i_R(\End_{\mathcal{X}_\Sigma}(\mathcal{T}),R)\quad \quad \forall i>0.
    \]
    Thus, with module-finiteness of $\Lambda$, we deduce that $\Lambda=\End_{\mathcal{X}_\Sigma}(\mathcal{T})\in \operatorname{CM}R$, and thus also $\Lambda\in\operatorname{ref}R$. Since also $R$ is noetherian, $\Lambda$ is necessarily supported everywhere, and so $i^\ast\mathcal{V}\neq 0$. $R_\mathfrak{p}$ is a local ring, and so the only perfect complexes $x$ with $\Hom_{\dbmod{R_\mathfrak{p}}}(x,x[i])=0$ for all $i>0$ are shifts of projective modules (Lemma 2.12, \cite{RZ03}).

    Thus $g_\ast i^\ast\mathcal{V}\cong R_\mathfrak{p}^a[b]$ for some $a\in\mathbb{N}$ and $b\in\Z$. Therefore,  $\Lambda_\mathfrak{p}\cong\End_{R_\mathfrak{p}}(R_\mathfrak{p}^a)$, which is Morita equivalent to $R_\mathfrak{p}$. As this is true for any height 1 prime $\mathfrak{p}\in\spec R$ and further $\Lambda$ is reflexive, we can apply Proposition \ref{Prop:Prop4.4IW} to obtain that $\Lambda\cong \End_R(M)$ as $R$-algebras for some $M\in\operatorname{ref}R$. Hence $\Lambda$ is indeed an NCCR, as claimed.
\end{proof}

\begin{remark}
    We note that this result, and its proof, are in line with expectations elaborated in \cite{VdB23}: $X_{\Sigma_\Delta}$ is a not necessarily crepant resolution of $\spec R$, and so we expect that the derived category of an NCCR embeds inside $\dbcoh{X_{\Sigma_\Delta}}$, as NCCRs should be minimal in a categorical sense. This sentiment is a strengthening of Van den Bergh's conjecture that all crepant resolutions, commutative or not, should be derived equivalent. We refer the reader to work of Kuznetsov \cite{Kuz08} for more results in this direction.
\end{remark}

\begin{remark}  
As remarked in Lemma \ref{Lem:siltimpliestilt}, a presilting complex with endomorphism algebra of finite global dimension on $\mathcal{X}_\Sigma$ here is in fact a tilting complex. We prefer the statement as written, since evaluating presilting complexes may be computationally more straightforward than checking the full tilting condition, since such objects could be obtained by mapping fully faithfully into $\dbcoh{\mathcal{X}_\Sigma}$ from another sufficiently well-understood category $D$.

In a certain sense, this Theorem combines the results by \v{S}penko-Van den Bergh and Iyama-Wemyss. Instead of needing a crepant map from a Gorenstein \textbf{variety} like in the case of Theorem \ref{Thm;Cor4.15IW}, we allow for a stack. Note that the map from $\mathcal{X}_\Sigma$ to $\spec R$ is still crepant. Meanwhile, Proposition \ref{Prop:3.3SVDBtoricII} allows for a crepant map from a stack $\mathcal{X}_\Sigma$ to $\spec R$, but requires the tilting object to be a bundle, whereas we allow for general tilting complexes. In this way, we can also hope to generate NCCRs in cases where no toric NCCR exists.
\end{remark}

\begin{remark}  
    Note that using the Theorem \ref{Prop:Parttiltingworks}, and not only the Corollary, it is now sufficient to know of partial tilting complexes with finite global dimension on another category $\mathcal{D}$ that maps fully faithfully into $\dbcoh{\mathcal{X}_\Sigma}$. So in particular if we have a homologically smooth category $D$ admitting a tilting object that maps fully faithfully into $\dbcoh{\mathcal{X}_\Sigma}$, we are guaranteed an NCCR as well.
\end{remark}

\section{Gorenstein cones via canonical bundles over projective varieties}
\label{sec:CanBdls}

\noindent We now examine how to use Theorem \ref{Prop:Parttiltingworks} as a tool to generate NCCRs. 
Let us restrict our attention to affine toric Gorenstein varieties whose associated cone $\sigma$ is the support of the canonical vector bundle on complete, simplicial toric varieties. Those cones are precisely those Gorenstein cones of the form $\cone(P\times\{1\})$ with $P$ a polytope with primitive vertices (i.e. $v=(v_1,\dots,v_n)$ such that $\gcd(v_i)=1$) and $0\in \operatorname{Int}(P)$. The way we generate NCCRs here will be to find tilting complexes on $\mathcal{X}_\mathcal{V}$, where $\mathcal{V}$ is the simplicial fan corresponding to the canonical bundle. In doing so, we triangulate $P$ on an interior point and no longer require the tilting object to be a bundle, exhibiting thus the advantage over using Proposition \ref{Prop:3.3SVDBtoricII}. Furthermore, so long as the fan is simplicial, we do not require it to be smooth, extending thus Theorem \ref{Thm;Cor4.15IW}.

The class of cones we consider here is quite restrictive, in particular requiring 0 to be an interior point. We can provide some leeway for this via the following Lemma.
\begin{lemma}
    \label{Lem:TranslationOfCone}
    Let $\sigma=\cone(P\times\{1\})$ and $\sigma'=\cone((P-n)\times\{1\})$ for $n\in N$. Then $X_\sigma\cong X_{\sigma'}$ and $R_\sigma\cong R_{\sigma'}$.
\end{lemma}
\begin{proof}
Translation by $n\in N$ is a unimodular shear of the cone, which preserves the toric variety and thus its coordinate ring.
\end{proof}
\noindent Thus, if $P$ contains an interior point, we may assume (by translating if necessary) that this interior point is $(0,\dots,0)$.

Let us first consider under which conditions we can find appropriate (partial) tilting complexes and, consequently, NCCRs for cones arising as support of toric vector bundles. Observe the following immediate consequence of Theorem \ref{Prop:Parttiltingworks}.

\begin{corollary}
\label{Cor:TiltCplxCanbdl}
    Let $X$ be a complete simplicial toric variety such that $\mathcal{V}:=[\tot \omega_X]$ admits a tilting complex. Consider the cone $\sigma=|\Sigma_V|$, where $\Sigma_V$ is the fan associated to the toric vector bundle $\tot \omega_X$. Then $R=k[\sigma^\vee\cap M]$ has an NCCR.
\end{corollary}

\begin{proof}
    We note that the cone $\sigma$ is, by construction, Gorenstein with respect to the element $(0,\dots,0,1)$. The star subdivision of $\sigma$ on $(0,\dots,0,1)$ corresponds to the toric stack $[\tot \omega_X]=\mathcal{V}$, giving a simplicial fan (as $X$ itself is simplicial). Hence, the associated toric DM stack $\mathcal{V}$ is smooth (Theorem \ref{Thm:4.12FK18}). By assumption, there is a tilting complex $\mathcal{T}$ on $\mathcal{V}$. Thus,  $\Lambda=\End_{\dbcoh{\mathcal{V}}}(\mathcal{T})$ has finite global dimension. Applying Theorem \ref{Prop:Parttiltingworks} yields the result.
\end{proof}

\begin{example}
\label{Exa:TotPn}
Consider the polyhedron \[P=\conv(e_1,e_2,\dots,e_n,-\sum_{i=1}^ne_i)\subseteq \R^n\] and the cone $\sigma=\cone(P\times\{1\})$. Note that $\sigma=|\Sigma_\mathcal{V}|$, where $\Sigma_\mathcal{V}$ is the fan associated to the total space of the vector bundle $\omega_{\P^n}$. Consider the standard tilting bundle on $\P^n$ due to Beilinson \cite{Bei78}, $\mathcal{T}_n=\O_{\P^n}\oplus\dots\oplus\O_{\P^n}(n)$. Write the canonical bundle as $\pi:X=\tot \O_{\P^n}(-n-1)\rightarrow \P^n$. Note that $X=\spec(\operatorname{Sym}^\bullet(\O_{\P^n}(n+1)))$ and $\O_X(l)=\pi^\ast\O_{\P^n}(l)$. Define the object $\mathcal{T}_X:=\pi^\ast \mathcal{T}_{\P^n}=\O_X\oplus\dots\oplus \O_X(n)$. This object $\mathcal{T}_X$ is tilting on $X$. Indeed, we have
\begin{align*}
\Rhom_X(\mathcal{T}_X,\mathcal{T}_X)&=\Rhom(\mathcal{T}_{\P^n},R\pi_\ast\pi^\ast\mathcal{T}_{\P^n})\\ 
&=\Rhom_{\P^n}(\mathcal{T}_{\P^n},\mathcal{T}_{\P^n}\otimes R\pi_\ast\O_X)\text{ by projection formula}\\
&=\bigoplus_{l\geq 0}\Rhom_{\P^n}(\mathcal{T}_{\P^n},\mathcal{T}_{\P^n}\otimes \O_{\P^n}(n+1)^{l})\\
&=\bigoplus_{l\geq 0}\Rhom_{\P^n}(\mathcal{T}_{\P^n},\mathcal{T}_{\P^n}\otimes \O_{\P^n}((n+1)l))\\
& =\Rhom(\mathcal{T}_{\P^n},\mathcal{T}_{\P^n})\oplus\bigoplus_{l\geq 1}\Rhom_{\P^n}(\mathcal{T}_{\P^n},\mathcal{T}_{\P^n}\otimes\O_{\P^n}((n+1)l))\;
\end{align*}
And so\begin{align*}
\Rhom^{>0}(\mathcal{T}_X,\mathcal{T}_X)&=\Rhom^{>0}(\mathcal{T}_{\P^n},\mathcal{T}_{\P^n})\oplus\bigoplus_{l\ge1}\Rhom^{>0}(\mathcal{T}_{\P^n},\mathcal{T}_{\P^n}\otimes\O_{\P^n}((n+1)l))\\
&=\bigoplus_{l\ge1}\Rhom^{>0}(\mathcal{T}_{\P^n},\mathcal{T}_{\P^n}\otimes\O_{\P^n}((n+1)l))\text{ as }\Ext^i(\mathcal{T}_{\P^n},\mathcal{T}_{\P^n})=0\text{ for }i>0.\\
&= \bigoplus_{l\ge 1} H^{>0}(\P^n,\mathcal{T}_{\P^n}^\vee\otimes\mathcal{T}_{\P^n}\otimes\O_{\P^n}((n+1)l))\\
&= \bigoplus_{l\geq1}H^{>0}\left(\P^n, \bigoplus_{0\le i,j\le n}\O_{\P^n}(i-j+(n+1)l)\right)\\
&=0 \text{ noting that }(n+1)l+i-j\ge1\text{ so cohomology vanishes}.
\end{align*}

Thus, we indeed have $\Rhom_X^{>0}(\mathcal{T}_X,\mathcal{T}_X)=0$ and hence $\mathcal{T}_X$ is partial tilting. To show it is tilting, we require $\Rhom(\mathcal{T}_X,\mathcal{F})=0\Rightarrow \mathcal{F}=0$. Note $0=\Rhom_X(\pi^\ast\mathcal{T}_{\P^n},\mathcal{F})=\Rhom_{\P^n}(\mathcal{T}_{\P^n},R\pi_\ast\mathcal{F})$. But as $\mathcal{T}_{\P^n}$ is tilting on $\P^n$, we thus have $R\pi_\ast\mathcal{F}=0$. As $\pi$ is affine and flat (since it is a vector bundle projection), we obtain that $\mathcal{F}=0$. Thus $\mathcal{T}_X$ is tilting and so, using Corollary \ref{Cor:TiltCplxCanbdl}, we obtain an NCCR for $R=k[\sigma^\vee\cap M]$.
\end{example}

It should be noted that Example \ref{Exa:TotPn} is one where the toric varieties involved are smooth, and so by Theorem \ref{Thm:4.12FK18} we have $[\tot\omega_X]\cong \tot\omega_X$. Hence, the existence of the NCCR is already covered by Theorem \ref{Thm;Cor4.15IW}. However, the strength of Corollary \ref{Cor:TiltCplxCanbdl} lies precisely in the fact that we do not require the toric varieties themselves to be smooth, but only simplicial, as the associated stack becomes the relevant smooth object. It is therefore desirable to find a way of producing tilting objects on toric DM stacks. A first case to look at is when these stacks take the form of global quotient stacks. This case has been studied by Novakovi\'{c} \cite{Nova18}, who gives sufficient conditions for the existence of tilting bundles of affine bundles over smooth projective varieties. We aim to generalise their proofs slightly to give us the corresponding statements for tilting complexes on complete, simplicial toric varieties and their associated affine bundles.

Let $G$ be a finite group acting on a smooth projective scheme $X$ and consider the total space $\mathbb{A}(\mathcal{E})=\spec(\operatorname{Sym}(\mathcal{E}))$, where $\mathcal{E}$ is an equivariant locally free sheaf of finite rank. For convenience, we will denote $\operatorname{Sym}^\bullet(\mathcal{E})$ by $S^\bullet(\mathcal{E})$. Note that $G$ acts on $\mathbb{A}(\mathcal{E})$, as $\mathcal{E}$ carries an equivariant structure and that there is an affine $G$-morphism $\pi:\mathbb{A}(\mathcal{E})\rightarrow X$. Novakovi\'{c} proved the following Theorem.
\begin{theorem}[=Theorem 5.1 in \cite{Nova18}]
    \label{Thm:NovaAffineBdlTiltBdl}
    With $X,G, \mathcal{E}$ as above, suppose that $\mathcal{T}$ is a tilting bundle on $[X/G]$. If $H^i(X,\mathcal{T}^\vee\otimes\mathcal{T}\otimes S^l(\mathcal{E}))=0$ for all $i\neq 0$ and all $l>0$, then $\pi^\ast\mathcal{T}$ is a tilting bundle on $[\mathbb{A}(\mathcal{E})/G]$.
\end{theorem}

\noindent By mildly adapting the proof Novakovi\'{c} provides, we obtain a slightly stronger result than Theorem \ref{Thm:NovaAffineBdlTiltBdl}.

\begin{theorem}
    \label{Thm:NovaGeneralCplx}
    With $X, G, \mathcal{E}$ as above, suppose that $\mathcal{T}$ is a tilting complex on $[X/G]$. If $H^i(X,\mathcal{T}^\vee\otimes \mathcal{T}\otimes S^l(\mathcal{E}))=0$ for all $i\neq 0$ and all $l>0$, then $\pi^\ast\mathcal{T}$ is a tilting complex on $[\mathbb{A}(\mathcal{E})/G]$.
\end{theorem}

\begin{proof}

Note that $\pi$ is a vector bundle projection, hence flat, and so $L\pi^\ast=\pi^\ast$ and $\pi^\ast\mathcal{T}$ is a perfect complex. In the following, we write $\Hom_G(\mathcal{F},\mathcal{G})$ to refer to the $G$-equivariant morphisms from $\mathcal{F}$ to $\mathcal{G}$, i.e. $\Hom_G(\mathcal{F},\mathcal{G})=\Rhomi{0}^G_{X}(\mathcal{F},\mathcal{G})$. 
We first want to show that $\pi^\ast \mathcal{T}$ is partial tilting, i.e. $\Ext^i_{}(\pi^\ast\mathcal{T},\pi^\ast\mathcal{T})=\Hom_G(\pi^\ast\mathcal{T},\pi^\ast\mathcal{T}[i])=0$ for $i\neq 0$. 
\begin{align*}
    \Hom_G(\pi^\ast \mathcal{T},\pi^\ast\mathcal{T}[i])  &\simeq \Hom_G(\mathcal{T},R\pi_\ast\pi^\ast\mathcal{T}[i])\quad\text{Adjunction of } L\pi^\ast, R\pi_\ast\\
     &\simeq \Hom_G(\mathcal{T},S^\bullet(\mathcal{E})\otimes \mathcal{T}[i])\quad\text{projection formula}\\
     & \simeq \Hom(\mathcal{T},S^\bullet(\mathcal{E})\otimes \mathcal{T}[i])^G\quad\text{equivariant morphisms}.\;
\end{align*}

To show that for $i\neq 0$ this is indeed zero, fix $l>0$ and observe\[
\Hom(\mathcal{T}, S^l(\mathcal{E})\otimes \mathcal{T}[i])\simeq\Ext^i(\mathcal{T}, S^l(\mathcal{E})\otimes \mathcal{T})\simeq H^i(X,\mathcal{T}^\vee\otimes \mathcal{T}\otimes S^l(\mathcal{E})).
\]
The first equivalence follows as $\Ext^i(\mathcal{F},\mathcal{G})=\Hom(\mathcal{F},\mathcal{G}[i])$, and the second equivalence follows as $\mathcal{T}$ is perfect, so dualisable, and thus 
$\Ext^i(\mathcal{T},S^l(\mathcal{E})\otimes\mathcal{T})=H^i(\Rshom(\mathcal{T},S^l(\mathcal{E}))\otimes \mathcal{T})=H^i(X,\mathcal{T}^\vee\otimes \mathcal{T}\otimes S^l(\mathcal{E}))$. 

By assumption, this is 0 for all $i\neq 0$ and $l>0$ and so the $G$-invariants are also 0. Thus, the higher $\Ext$ vanish and $\pi^\ast\mathcal{T}$ is partial tilting. 

Now assume $\Rhom(\pi^\ast\mathcal{T},\mathcal{F})=0$ for $\mathcal{F}\in\dbcoh{[\mathbb{A}(\mathcal{E})/G]}$. By adjunction, this implies $\Rhom(\mathcal{T},R\pi_\ast \mathcal{F})=0$. But $\mathcal{T}$ is tilting on $[X/G]$, and so $R\pi_\ast\mathcal{F}=0$. Since $\pi$ is affine, this implies $\mathcal{F}=0$ and thus $\pi^\ast\mathcal{T}$ is tilting on $[\mathbb{A}(\mathcal{E})/G]$, as desired.
\end{proof}

Combining this with Theorem \ref{Prop:Parttiltingworks} immediately gives the following.

\begin{corollary}
    \label{Cor:Nova+myresult}
    Let $X_\Sigma$ be a projective simplicial toric variety and consider the cone $\sigma=|\Sigma_\mathcal{V}|$, where $\Sigma_\mathcal{V}$ is the fan associated to the toric vector bundle $\tot \omega_{X_\Sigma}$. Suppose $\mathcal{X}_{\Sigma}\cong [X/G]$ for some smooth, projective $X$ and a finite group $G$, that $\mathcal{T}$ is tilting on $[X/G]$ and that $[\tot \omega_{X_\Sigma}]=[\tot \mathcal{E}^\vee/G]$ for some $G$-equivariant vector bundle $\mathcal{E}$. If $H^i(X,\mathcal{T}^\vee\otimes \mathcal{T}\otimes \operatorname{S}^l(\mathcal{E}))=0$ for $i\neq 0, l>0$, then $R=k[\sigma^\vee\cap M]$ has an NCCR.
\end{corollary}

Note that for a finite group $G$ acting on a smooth projective $X$, there is a $G$-morphism $f:X\rightarrow \spec k$, where $G$ acts trivially on $\spec k$. For $W$ an irreducible representation of $G$, note that the sheaf $f^\ast W=\O_X\otimes W$ has a natural $G$-equivariant structure. By abuse of notation, we will write $W$ for $f^\ast W$ when the context is clear. We examine Theorem 4.1 in \cite{Nova18}, which combines particularly well with Theorem \ref{Thm:NovaGeneralCplx}.
\begin{theorem}[=Theorem 4.1 \cite{Nova18}]
    \label{Thm:Nova4.1}
    Let $X$ be a smooth projective scheme and $G$ a finite group acting on $X$. Suppose there is a $\mathcal{T}\in \operatorname{D}(\operatorname{Qcoh}([X/G]))$ which, considered as object in $\operatorname{D}(\operatorname{Qcoh}(X))$, is tilting on $X$.  Denoting by $W_j$ the irreducible representations of $G$, the object $\mathcal{T}_G:=\bigoplus_j T\otimes W_j$ is tilting on $[X/G]$.
\end{theorem}

We also wish to consider partial tilting complex with finite global dimension. The proof of Theorem \ref{Thm:Nova4.1} does still apply to a partial tilting object; the finite global dimension, however, is more subtle.

\begin{theorem}
    \label{Thm:parttiltNova4.1}
    Consider $X,G$ as in Theorem \ref{Thm:Nova4.1} and $\mathcal{T}$ to be partial tilting with $\gldim\End_X(\mathcal{T})<\infty$.
    Then $\mathcal{T}_G$ as in Theorem \ref{Thm:Nova4.1} is partial tilting and $\gldim\End_{[X/G]}(\mathcal{T}_G)<\infty$.
\end{theorem}
\begin{proof}
    The same proof as for Theorem \ref{Thm:Nova4.1} applies to show that $\mathcal{T}_G$ is partial tilting. 

    Consider $\mathcal{T}'=\mathcal{T}\otimes k[G]$, where $k[G]$ is the regular representation. It is known that $\End_{[X/G]}(\mathcal{T}')\cong \Lambda\rtimes G$ (see, e.g., \cite[Lemma 2.6]{KN25}).
    
  \noindent  Write, for sake of simplicity, $A=\Lambda\rtimes G$ and $d=\gldim\Lambda$. $A$ is a free left and right $\Lambda$-module. Since we work over a field $k$ of characteristic 0 $\Lambda\subset A$ is separable, and hence for any left $A$ module $M$, $M$ is an $A$-module direct summand of $A\otimes_{\Lambda}M$. Considering $M$ as $\Lambda$-module, we can choose a projective resolution of $M$ in terms of projective $\Lambda$-modules which is of length $\le d$. $A$ is free as right $\Lambda$ module and so applying $A\otimes_\Lambda -$ is exact, which gives an $A$-projective resolution of $A\otimes_\Lambda M$ of length $\le d$. Separability means that $M$ is a direct summand, and so $\operatorname{pd}M\le d$. Hence, $\gldim A\le\gldim\Lambda<\infty$. It remains to show that this implies finite global dimension of $\End_{[X/G]}(\mathcal{T}_G)$.

    By Maschke's theorem, we have $k[G]\cong \bigoplus_{j}W_j^{\dim W_j}$ with the sum over all irreducible representations of $G$. But then $\operatorname{add}\mathcal{T}_G=\operatorname{add}\mathcal{T}'$, and thus $A$ and $\End_{[X/G]}(\mathcal{T}_G)$ are Morita equivalent - proving the finiteness of the global dimension we required.

\end{proof}

We conjecture further that Theorem \ref{Thm:NovaAffineBdlTiltBdl} and Corollary \ref{Cor:Nova+myresult} also generalise in this manner.
\begin{conjecture}
    \label{Thm:BdlParTiltNova}
    With $X, G, \mathcal{E}$ as above, suppose that $\mathcal{T}$ is a partial tilting complex on $[X/G]$ such that $\gldim \End(\mathcal{T})<\infty$. If $H^i(X,\mathcal{T}^\vee\otimes \mathcal{T}\otimes S^l(\mathcal{E}))=0$ for all $i\neq 0$ and all $l>0$, then $\pi^\ast\mathcal{T}$ is a partial tilting complex on $[\mathbb{A}(\mathcal{E})/G]$ such that $\gldim\End(\pi^\ast\mathcal{T})<\infty$.
\end{conjecture}

\begin{conjecture}
    \label{Cor:Nova+myresultpart}
    Let $X_\Sigma$ be a simplicial projective toric variety and consider the cone $\sigma=|\Sigma_\mathcal{V}|$, where $\Sigma_\mathcal{V}$ is the fan associated to the toric vector bundle $\tot \omega_{X_\Sigma}$. Suppose $\mathcal{X}_\Sigma=[X/G]$ for $X$ a smooth, projective variety and $G$ a finite group such that $[\tot \omega_{X_\Sigma}]\cong [\mathbb{A}(\mathcal{E})/G]$ for a $G$-equivariant vector bundle on $X$. Suppose $\mathcal{T}_G$ is partial tilting on $[X/G]=\mathcal{X}_\Sigma$ such that $\gldim\End_{[X/G]}(\mathcal{T}_G)<\infty$. If $H^i(X,\mathcal{T}^\vee\otimes \mathcal{T}\otimes S^l(\mathcal{E}))=0$ for $i\neq 0, l>0$, then $R=k[\sigma^\vee\cap M]$ has an NCCR.
\end{conjecture}

\noindent The results of Novakovi\'{c} require the variety $X$ to be smooth projective, but do not require the variety to be toric. In our case, we can substitute smoothness of the underlying toric variety by considering the toric DM stack associated to it. Similarly, projectivity is a stronger condition than necessary - properness suffices and so we should examine complete, simplicial fans.

\begin{theorem}
    \label{Thm:PullbackDMstack}
Let $\Sigma$ be a simplicial fan and consider a vector bundle $V=\bigoplus_{j=1}^r\O_{X_\Sigma}(D_j)$, with a fan $\Sigma_V\subset (N\oplus \Z^r)_{\R}$ (constructible using \cite[\S 7]{CLS}). Then $X_{\Sigma_V}\cong\tot_{X_\Sigma}(V)$ and the lattice projection $N\oplus \Z^r\rightarrow N$ induces the bundle projection $p:X_{\Sigma_V}\rightarrow X_\Sigma$. Consider now the good moduli spaces $f_\Sigma:\mathcal{X}_\Sigma\rightarrow X_\Sigma$ and $f_{\Sigma_V}:\mathcal{X}_{\Sigma_V}\rightarrow X_{\Sigma_V}$. Then we have $\mathcal{X}_{\Sigma_V}\cong\tot_{\mathcal{X}_\Sigma}(f_\Sigma^\ast V)$ and a vector bundle projection $\pi:\mathcal{X}_{\Sigma_V}\rightarrow \mathcal{X}_\Sigma$. Suppose $\mathcal{X}_\Sigma$ admits a tilting complex $\mathcal{T}$. If $H^i(\mathcal{X}_\Sigma,\mathcal{T}^\vee\otimes\mathcal{T}\otimes \operatorname{Sym}^l(f_\Sigma^\ast V^\vee))=H^i(X_\Sigma,f_{\Sigma,\ast}(\mathcal{T}^\vee\otimes\mathcal{T})\otimes \operatorname{Sym}^l(V^\vee))=0$ for all $i\neq 0$ and $l\ge 1$, then $\pi^\ast\mathcal{T}$ is tilting on $\mathcal{X}_{\Sigma_V}$.
\end{theorem}
\begin{proof}
Note firstly that we have a 2-cartesian square 
\[
\begin{tikzcd} \mathcal X_{\Sigma_V} \arrow[r,"f_{\Sigma_V}"] \arrow[d,"\pi"'] & X_{\Sigma_V} \arrow[d,"p"] \\ \mathcal X_\Sigma \arrow[r,"f_\Sigma"'] & X_\Sigma . \end{tikzcd}
\]
This induces $\mathcal{X}_{\Sigma_V}\simeq \mathcal{X}_\Sigma\times_{X_\Sigma}X_{\Sigma_V}$. Further, since $\pi$ is affine that gives $R^j\pi_\ast=0$ for $j>0$ and so $R\pi_\ast\O_{\mathcal{X}_{\Sigma_V}}=\pi_\ast\O_{\mathcal{X}_{\Sigma_V}}=\operatorname{Sym}^\bullet(f_\Sigma^\ast V^\vee)$. 

Now \begin{align*}
    \Ext^i_{\mathcal{X}_{\Sigma_V}}(\pi^\ast\mathcal{T},\pi^\ast\mathcal{T})&\simeq \Hom_{\mathcal{X}_\Sigma}(\mathcal{T},R\pi_\ast\pi^\ast\mathcal{T}[i])\\
    &\simeq H^i(\mathcal{X}_\Sigma,\mathcal{T}^\vee\otimes \mathcal{T}\otimes R\pi_\ast\O_{\mathcal{X}_{\Sigma_V}})\\
    &\simeq H^i(\mathcal{X}_\Sigma,\mathcal{T}^\vee\otimes\mathcal{T}\otimes \operatorname{Sym}^\bullet(f_\Sigma^\ast V^\vee))\\
    &\simeq \bigoplus_{l\ge 0}H^i(\mathcal{X}_\Sigma,\mathcal{T}^\vee\otimes \mathcal{T}\otimes \operatorname{Sym}^l(f_\Sigma^\ast V^\vee)).
    \end{align*}

    By hypotheses and tilting property, the cohomologies on the right side vanish, which gives $\pi^\ast\mathcal{T}$ the required tilting property. It should be specified here that $L\pi^\ast\mathcal{T}=\pi^\ast \mathcal{T}$ because $\pi$ is a vector bundle projection, and hence flat. To show that $\pi^\ast\mathcal{T}$ generates, we note that $\Rhom_{\mathcal{X}_{\Sigma_V}}(\pi^\ast\mathcal{T},F)=0$ for $F\in \operatorname{D}_{Qch}(\mathcal{X}_{\Sigma_V})$ gives, by adjunction, $\Rhom_{\mathcal{X}_\Sigma}(\mathcal{T},R\pi_\ast F)=0$. Since $\mathcal{T}$ was tilting on $\mathcal{X}_\Sigma$, we then have $R\pi_\ast F=0$. As we are pushing forward along an affine morphism, that means $F=0$. 
\end{proof}

\begin{corollary}
    \label{Cor:NovaDMstack}
In the situation of Theorem~\ref{Thm:PullbackDMstack}, if $V=\omega_{X_\Sigma}$, then set $N'=N\oplus \Z$ and $M'=M\oplus \Z$. Denote $\sigma=|\Sigma_V|\subseteq N'_{\R}$. Suppose $\sigma$ is a full-dimensional strongly convex rational polyhedral cone and that the induced toric morphism \[
q:X_{\Sigma_V}\rightarrow X_\sigma
\]
is projective (i.e. that the triangulation of $\sigma_{(1)}$ is regular). Then $R=k[\sigma^\vee\cap M']$ has an NCCR.
\end{corollary}

\begin{proof}
    It is easy to see that $\sigma$ is Gorenstein (as all generators lie in the hyperplane $\{n\in N'\mid\langle (0,\dots,0,1),n\rangle=1\}$) and the subdivision given by $\Sigma_V$ is naturally crepant. By Theorem\ref{Thm:PullbackDMstack}, $\pi^\ast\mathcal{T}$ is tilting on $\mathcal{X}_{\Sigma_V}$, which is smooth as $\Sigma_V$ is simplicial. Thus, Theorem \ref{Prop:Parttiltingworks} gives the desired result.
\end{proof}

The results we presented are helpful when the Gorenstein cone $\cone(P\times\{1\})$ lies over a polytope $P$ with interior points that, when one of them is translated to zero, admits primitive vertices and a simplicial face fan, after which we are required to find a tilting object. In practice, even simplicity often fails. As an example, consider the cube in $\R^3$ -  its face fan contains the cone $\cone((1,1,1),(1,-1,1),(-1,1,1),(-1,-1,1))$, which is not simplicial. In the case of cones over two-dimensional polytopes, however, the face fan always gives a simplicial projective toric variety (since all two-dimensional cones are simplicial), and thus a smooth toric DM stack. To circumvent the issue of being non-simplicial, one can sometimes find simplicial subdivisions of the face fan. Furthermore, the results as they are can still be useful, as we will demonstrate in $\S $\ref{sec:refl}, where we discuss simplicial reflexive polytopes.

\begin{example}
\label{Exa:Fabricated}
    Consider the polytope \[P=\conv((2,2,3),(0,2,3),(1,3,3),(1,1,3),(2,3,6),(0,1,0))\] and consider the cone $\sigma=\cone(P\times\{1\})$. The polytope $P$ contains the interior point $m=(1,2,3)$. Shifting by $m$ gives the polytope \[P'=\conv((1,0,0),(-1,0,0),(0,1,0),(0,-1,0),(1,1,3),(-1,-1,-3))\] with associated cone $\sigma'=\cone(P'\times\{1\})$. Consider the face fan $\Sigma$ of $P'$. One finds that $\mathcal{X}_{\Sigma}\cong [(\P^1)^3/(\Z/3\Z)]$. The total space of the canonical bundle $\tot \omega_{\mathcal{X}_{\Sigma}}$ corresponds to a star subdivision of $\sigma'$ on $(0,0,0,1)$. A computation shows that $[\tot\omega_{\mathcal{X}_\Sigma}]\cong[\tot\omega_{(\P^1)^3}/(\Z/3\Z)]$. The standard tilting bundle $\bigoplus_{l_i\in \{0,1\}} \O(l_1,l_2,l_3)$ on $(\P^1)^3$ gives, via Theorem \ref{Thm:Nova4.1}, a tilting bundle $\mathcal{T}_G$ on $\mathcal{X}_{\Sigma}$, which lifts to the canonical bundle. It remains to check that $H^i((\P^1)^3, \mathcal{T}_G^\vee\otimes\mathcal{T}_G\otimes\O(2l,2l,2l))=0$ for $i\neq 0, l>0$.

\end{example}

\begin{example}
\label{Exa:FMSviaCanonical}
    Consider the cone $\sigma=\cone((2,2,1),(2,0,1),(0,0,1),(0,-2,1))$. We start by shifting the underlying polytope $P=\conv((2,2),(2,0),(0,0),(0,-2))$ by $(-1,0)$ to obtain the, by Lemma \ref{Lem:TranslationOfCone} equivalent, situation of \[\sigma'=\cone((-1,0,1),(1,2,1),(-1,-2,1),(1,0,1)).\] Denote the lattice corresponding to the first two coordinates by $N$. Now consider the index 2 sublattice $\Z\langle(1,0),(1,2)\rangle$, denoted by $N'$. There, the face fan of the polytope $\conv((-1,0),(1,0),(1,2),(-1,-2))$ becomes the natural fan for $\P^1\times \P^1$ and a computation confirms that in this case, $\mathcal{X}_{\Sigma_P,N}\cong[\P^1\times\P^1/(\Z/2\Z)]$. 
    
    \noindent The standard tilting bundle $\mathcal{T}$ on $\P^1\times \P^1$ allows for a $(\Z/2\Z)$-equivariant structure, and so we obtain a tilting object $\mathcal{T}_G=\mathcal{T}\otimes(\O_{\P^1\times\P^1} \otimes W_1)\oplus \mathcal{T}\otimes (\O_{\P^1\times\P^1} \otimes W_2)$, where $W_1,W_2$ are the irreps of $\Z/2\Z$. We note that the ray generators of the toric vector bundle $\tot \omega_{X_{\Sigma_P,N}}$ are $(1,0,1), (-1,-2,1), (-1,0,1), (1,2,1)$ and $(0,0,1)$ and are primitive in both lattices. Thus, we can compute the associated toric DM stack to obtain the isomorphism $[\tot\omega_{X_{\Sigma_P.N}}]\cong [\tot\omega_{\P^1\times\P^1}/(\Z/2\Z)]$ and $\tot \omega_{\P^1\times\P^1}=\mathbb{A}(\O(-K_{\P^1\times\P^1}))=\mathbb{A}(\O_{\P^1\times\P^1}(2,2))$. 
    One verifies $H^i(\P^1\times\P^1, \mathcal{T}_G^\vee\otimes \mathcal{T}_G\otimes \O(2l,2l))=0$ for all $i\neq 0, l>0$, and thus Corollary \ref{Cor:NovaDMstack} gives an NCCR for $R=k[(\sigma')^\vee\cap M]\cong k[\sigma^\vee\cap M]$.
\end{example}

    The above example is closely related to Example 7.10 in \cite{FMS19}. In fact, the cone $C$ given in their example is precisely the cone $\cone(\frac{1}{2}P\times \{1\})$. In their work, Faber-Muller and Smith construct an NCCR for the toric variety associated to the cone $C^\vee$ (i.e. for the algebra $k[C\cap M]$). This leads to the following question.
    \begin{question}
        \label{Qn:DualNCCR}
        If a Gorenstein cone $\sigma$ has dual cone $\sigma^\vee$ also Gorenstein, does the existence of an NCCR for $k[\sigma^\vee\cap M]$ imply the existence of an NCCR for $k[\sigma\cap N]$?
    \end{question}
    Furthermore, the appearance of the factor of $\frac{1}{2}$, here obtained by modifying the lattice (our example can be obtained from theirs by refining the lattice $N$ to the index 4 sublattice $\frac{1}{2}N$), suggests another question:
    \begin{question}
    \label{Qn:MultipleOfPoly}
    Suppose $P$ is a lattice polytope. Is it true that $\sigma=\cone(P\times\{1\})$ has an NCCR if and only if $\sigma'=\cone(aP\times\{1\})$ does for $a\in \Z_{>0}$?
    \end{question}

\section{The reflexive case}
\label{sec:refl}

Let us conclude the discussion of Gorenstein cones arising as canonical bundles of projective varieties by considering the case of reflexive Gorenstein cones. These cones are natural to consider apart from other, as they correspond to Gorenstein toric Fano varieties.
We start by proving the following theorem, showing that reflexive Gorenstein cones obtained as cones over simplicial, reflexive polytopes with at most $\dim P+2$ vertices allow for an NCCR of the associated toric algebra.

\begin{theorem}
    \label{Thm:FanoAlmostSimplicial}
    Let $P\in N_{\R}\cong \R^n$ be a simplicial, reflexive polytope with $\le n+2$ vertices. Consider the cone $\sigma=\cone(P\times\{1\})$. Then $R=k[\sigma^\vee\cap M]$ has an NCCR.
\end{theorem}
\begin{proof}
As $\sigma$ is the cone over a reflexive polytope, it is reflexive Gorenstein of index 1 (see Proposition 1.11 \cite{BN07}). 
Consider the face fan $\Sigma_P$ of $P$. By construction, it is a complete, simplicial fan. Since the fan is simplicial, by Theorem \ref{Thm:4.12FK18} the toric DM stack is smooth. Furthermore, we note that $P$ is reflexive, and so $X_{\Sigma_P}$ is Fano. Hence, Theorem 5.11 in \cite{BH09} guarantees the existence of a tilting bundle on $\mathcal{X}_{\Sigma_P}$, as $\operatorname{rk}\operatorname{Pic}(\mathcal{X}_{\Sigma_P})\le 2$.  
By Theorem \ref{Thm:PullbackDMstack}, the tilting bundle $\mathcal{T}$ pulls back to a tilting bundle on $[\tot \omega_{X_\Sigma}]$ if and only if $H^i(\mathcal{X}_{\Sigma},\mathcal{T}^\vee\otimes\mathcal{T}\otimes (\omega_{\mathcal{X}_\Sigma}^\vee)^{\otimes l})=0$ for $l>0, i\neq0$. 
To prove this, we study the form of the tilting bundle given in \cite{BH09}. Borisov and Hua define convex sets called \newterm{forbidden cones} in $\operatorname{Pic}_\R(\mathcal{X}_\Sigma)=\operatorname{Pic}(\mathcal{X}_\Sigma)\otimes \R$ and show that a line bundle whose image in $\operatorname{Pic}_\R(\mathcal{X}_\Sigma)$ does not lie in any forbidden cone is acyclic.

If $\operatorname{rk}\operatorname{Pic}(\mathcal{X}_\Sigma)=1$, the single forbidden cone is given by $x\in \operatorname{Pic}_\R(\mathcal{X}_\Sigma)$ such that $\deg(x)\le -\sum_{\rho\in \Sigma(1)} \deg(D_\rho)=\deg(K_{\Sigma})$, where $K_\Sigma$ is the canonical divisor. Here, the degree function takes value 1 on the positive generator of the Picard group.
Borisov and Hua go on to prove that their full strong exceptional collection of line bundles $\{\mathcal{L}_i\}$, which we can build the tilting bundle out of, has the property that $\mathcal{L}_i^\vee\otimes\mathcal{L}_j$ has degree $>\deg(K_\Sigma)$ for any $i,j$. We note that $\mathcal{L}_i^\vee\otimes \mathcal{L}_j\otimes \O(l\sum_{\rho\in\Sigma(1)}D_\rho)$ has degree $\deg(\mathcal{L}_i^\vee\otimes \mathcal{L}_j)+l|\Sigma(1)|$ and so the line bundle still lies outside the unique forbidden cone. Hence, all direct summands of $\mathcal{T}^\vee\otimes\mathcal{T}\otimes (\omega_{\mathcal{X}_\Sigma}^\vee)^{\otimes l}$ are acyclic, and so the cohomology vanishes as required.

In the case of $\operatorname{rk}\operatorname{Pic}(\mathcal{X}_\Sigma)=2$, the proof changes slightly, but follows the same idea. In this case, there are exactly three forbidden cones, which we will name $F_\emptyset, F_+$ and $F_-$, which are defined via three subsets of $\Sigma(1)$: the empty set as well as two sets $I_+, I_-$ forming a partition of $\Sigma(1)$. There is a parallelogram $Q\ni 0$ such that no points in the interior of $2Q$ lie in the forbidden cones, and the boundary of $2Q$ contains exactly three points that lie in the forbidden cones, $-\sum_{\rho \in \Sigma(1)}D_\rho$, $-\sum_{\rho\in I_+}D_\rho$ and $-\sum_{\rho\in I_-}D_\rho$. In fact, one pair of parallel sides of $2Q$ gives supporting hyperplanes $H_-, H_+$ for the forbidden cones $F_-$ and $F_+$.
The full strong exceptional collection of line bundle $\{\mathcal{L}_i\}$ is defined as those line bundles whose image in $\operatorname{Pic}_\R(\mathcal{X}_\Sigma)$ lies in $Q+p$ for some generic point $p$ with the property that $Q+p$ contains no point of $\operatorname{Pic}(\mathcal{X}_\Sigma)\otimes \Q$. All bundles of the form $\mathcal{L}_i^\vee\otimes \mathcal{L}_j$ by construction lie in the interior of $2Q$ and thus not in any forbidden cone, hence are acyclic. 

\noindent However, more is true: we note that $\mathcal{L}_i^\vee\otimes\mathcal{L}_j\otimes \O(l\sum_{\rho\in\Sigma(1)} D_\rho)$ remains outside the forbidden cones as well. Indeed, the forbidden cone $F_\emptyset$ is defined as $-\sum_{\rho\in \Sigma(1)}(1+\R_{\ge0})D_\rho$, and so adding $l\sum_{\rho\in\Sigma(1)}D_\rho$ cannot translate $\mathcal{L}_i^\vee\otimes \mathcal{L}_j$ into the forbidden cone $F_\emptyset$.
Furthermore, adding $l\sum_{\rho\in\Sigma(1)}D_\rho$ does not affect on which hyperplane parallel to the supporting hyperplanes $H_{\pm}$ the line bundle lies. Hence, $\mathcal{L}_i^\vee\otimes\mathcal{L}_j\otimes \O(l\sum_{\rho\in\Sigma(1)}D_\rho)$ lies between, but not in, the two forbidden cones $F_{\pm}$. Since it lies in none of the forbidden cones, it is acyclic. This is true for all components of $\mathcal{T}^\vee\otimes \mathcal{T}\otimes (\omega_{\mathcal{X}_\Sigma}^\vee)^{\otimes l}$, and so the cohomology vanishes as required. By Corollary \ref{Cor:NovaDMstack}, $R=k[\sigma^\vee\cap M]$ thus has an NCCR, as required.
\end{proof}
\begin{remark}
    In light of Lemma \ref{Lem:TranslationOfCone}, we also obtain NCCRs for (almost) simplicial Gorenstein cones which are reflexive with respect to other lattice points than $(0,\dots,0,1)$
\end{remark}

The vanishing condition here was deduced by using the explicit construction of the tilting bundle on the smooth DM stack $\mathcal{X}_{\Sigma_P}$. However, we suspect that vanishing of higher cohomology of $\mathcal{T}^\vee\otimes \mathcal{T}\otimes (\omega^\vee)^{\otimes l}$ is a property of toric Fanos for tilting objects $\mathcal{T}$, and not of the specific tilting bundle constructed by Borisov and Hua. In the example of the tilting bundle, we could show strong acyclicity of $\O(F_j-F_i+\sum D_\rho)$ for any direct summands $\mathcal{L}_i=\O(F_i)$, $\mathcal{L}_j=\O(F_j)$ of $\mathcal{T}$. The intuition we have tells us that a similar condition holds in general: any component of $\mathcal{T}^\vee\otimes \mathcal{T}$ and their dual are acyclic and we suggest that twisting by $(\omega^\vee)^{\otimes l}$ preserves that property in the Fano case. 
Thus we formulate the following conjecture.

\begin{conjecture}
    \label{Conj:FanoVanishingII}
    Let $P$ be a reflexive polytope. Consider a simplicial fan $\Sigma$ such that the primitive generators of the rays $\rho\in\Sigma(1)$ are the vertices of $P$. If $\mathcal{X}_\Sigma$ has a tilting object $\mathcal{T}$, then so does $\mathcal{X}_\mathcal{V}$, where $\mathcal{V}$ is the fan of the canonical bundle over $X_\Sigma$. Thus, if $\mathcal{X}_\Sigma$ has a tilting object, $R=k[|\cone(P\times\{1\})|^\vee\cap M]$ has an NCCR.
\end{conjecture}

If the reflexive polytope $P$ is simplicial, its face fan is already simplicial and hence the associated toric DM stack is smooth (Theorem \ref{Thm:4.12FK18}). This describes smooth Fano DM stacks as discussed in Borisov and Hua's paper \cite{BH09}, so one may hope that their methods extend to prove tilting bundles or complexes for such $\mathcal{X}_\Sigma$. This gives the following special case of Conjecture \ref{Conj:FanoVanishingII}.

\begin{conjecture}
    \label{Conj:FanoVanishing}
    Let $P$ be a simplicial, reflexive polytope. Consider the face fan $\Sigma$ of $P$. If $\mathcal{X}_\Sigma$ has a tilting object $\mathcal{T}$, then so does $\mathcal{X}_\mathcal{V}$, where $\mathcal{V}$ is the fan of the canonical bundle over $X_\Sigma$. Thus, if $\mathcal{X}_\Sigma$ has a tilting object, $R=k[|\cone(P\times\{1\})|^\vee\cap M]$ has an NCCR.
\end{conjecture}
\begin{remark}
    In the case of a smooth toric Fano with a tilting bundle, this is a consequence of Theorem \ref{Cor:TiltCplxCanbdl}. Following Theorem 5.1 in \cite{Nova18}, the author notes the cohomology vanishing condition for smooth Fanos. This is also reflected by Theorem 3.6 in \cite{BS10}. It is further known that if a smooth projective Fano variety $X$ admits a tilting object of minimal possible global dimension, then the pullback to the canonical bundle admits a tilting object and an NCCR of the corresponding cone (see \S 1.2 in \cite{IdPKW}).
\end{remark}

\noindent It is a commonly known fact that the isomorphism classes of reflexive polytopes are in bijection with the isomorphism classes of Gorenstein toric Fano varieties by associating to a reflexive polytope the toric variety of its face fan (see, e.g., \cite{Casa06}). As such, the cases considered in Conjecture \ref{Conj:FanoVanishingII} correspond to simplicialisations of toric Fano varieties. We note the following result, elucidating their nature a little.
\begin{lemma}
    \label{Lem:SimplFanoisWeakFano}
    Let $\Sigma$ be the fan of a toric Fano variety. Consider a fan $\Sigma'$ obtained by simplicially subdividing the cones of $\Sigma$ without adding additional rays. Then $X_{\Sigma'}$ is weak Fano.
\end{lemma}

\begin{proof}
    A compact variety is weak Fano if its anticanonical divisor is nef and big. Since $|\Sigma'|=|\Sigma|$, the variety is compact (as the fan is complete). The anticanonical divisor $-K_{X_{\Sigma'}}$ is given by $\sum_{\rho\in \Sigma'(1)} D_\rho$ and it is associated to a polytope $P_{-K_{X_{\Sigma'}}}=\{m\in M_\R\mid \langle m, u_\rho\rangle\ge -1\}$. On a complete toric variety $X_{\Sigma'}$, the support function of a divisor $D$ is given by $\varphi_D(u)=\min_{m\in P_D}\langle m,u\rangle$ (Theorem 6.1.7 in \cite{CLS}) and so the support function of the anticanonical divisor on $X_{\Sigma'}$ is the same as on $X_\Sigma$. Since $-K_{X_\Sigma}$ is ample (as $X_\Sigma$ is Fano), it is in particular nef. The support of $\Sigma$ is convex and thus Lemma 9.2.1 in \cite{CLS} gives that the divisor $-K_{X_{\Sigma}}$ being nef is equivalent to the support function $\varphi_{-K_{X_\Sigma}}:|\Sigma|\rightarrow \R$ being convex. As $\varphi_{-K_{X_\Sigma}}=\varphi_{-K_{X_\Sigma'}}$, Lemma 9.2.1 shows that $-K_{X_{\Sigma'}}$ is nef on $X_{\Sigma'}$. A nef divisor $D$ on a complete toric variety $X_\Phi$ is big if and only if $\dim P_D=\dim X_\Phi$, and so it suffices to show that $\dim P_{-K_{X_{\Sigma'}}}=\dim X_{\Sigma'}=\dim X_\Sigma$. By definition, $P_{-K_{X_\Sigma'}}=P_{-K_{X_\Sigma}}$, and the statement follows. Thus, the anticanonical divisor on $X_{\Sigma'}$ is nef and big, and thus the variety is weak Fano, as claimed.
\end{proof}

Knowing this, and in light of previous results, we pose the following question.

\begin{question}
    \label{Qn:FanoSimplHasTilting}
    Let $X_\Sigma$ be a toric Fano variety. Is there always a simplicial toric weak Fano variety $X_{\Sigma'}$ such that $\Sigma(1)=\Sigma'(1)$ and that $\mathcal{X}_{\Sigma'}$ admits a tilting complex?
\end{question}

For the remainder of the paper, let us elaborate why we suspect that verifying Conjecture \ref{Conj:FanoVanishingII} provides a significant step towards proving Conjecture \ref{Conj:affinetoric}. To reduce the general case of Gorenstein cones to this, we begin by observing the following, well-known fact. 

\begin{proposition}[=Proposition 2.2 in \cite{HM06}]
    \label{Prop:FaceOfReflexive}
    Let $P$ be a lattice polytope. Then $P$ is lattice equivalent to a face of some reflexive polytope $Q$.
\end{proposition}

Thus, we know that any Gorenstein cone $\cone(P\times\{1\})$ appears (embedded into a bigger lattice) as a face of a reflexive Gorenstein cone $\cone(Q\times\{1\})$. We expect that providing an NCCR for the toric algebra associated to $\cone(P\times \{1\})$ helps in constructing an NCCR for the toric algebra associated to $\cone(Q\times \{1\})$, thus reducing Conjecture \ref{Conj:affinetoric} to the case of reflexive Gorenstein cones. 

\begin{conjecture}
    \label{Conj:FaceGivesNCCR}
    Let $Q\subset \R^d$ be a lattice polytope with associated Gorenstein cone $\sigma=\cone(Q\times\{1\})\subset \R^{d+1}$. Suppose $Q$ is lattice equivalent to a face $F$ of a lattice polytope $P\subset \R^n$ such that there exists an NCCR for $R_{\sigma'}=k[(\sigma')^\vee\cap M']$, where $\sigma'=\cone(P\times\{1\})\subset \R^{n+1}$ and $M'$ is the $(n+1)$-dimensional character lattice of $X_{\sigma'}$. Then there exists an NCCR for $R_\sigma=k[\sigma^\vee\cap M]$.
\end{conjecture}

\begin{remark}
    Progress has been made towards proving this conjecture and will be the subject of upcoming work. 
\end{remark}

To conclude this paper, we elaborate how positively answering Question \ref{Qn:FanoSimplHasTilting} and the pair of conjectures \ref{Conj:FanoVanishingII} and \ref{Conj:FaceGivesNCCR} provides a proof of Van den Bergh's conjecture (Conjecture \ref{Conj:affinetoric}). This will thus provide a guide for future research.

\begin{proposition}
    \label{Thm:IFQnandConjHAVEvdB}
    Given Conjecture \ref{Conj:FanoVanishingII} and \ref{Conj:FaceGivesNCCR}, if the answer to Question \ref{Qn:FanoSimplHasTilting} is positive, then for any Gorenstein cone $\sigma$, $R=k[\sigma^\vee\cap M]$ admits an NCCR.
\end{proposition}
\begin{proof}
Write the cone $\sigma$ as $\sigma=\cone(Q\times\{1\})$. By Proposition \ref{Prop:FaceOfReflexive}, there exists a reflexive polytope $P$ such that $Q$ is lattice equivalent to a face $F$ of $P$. Write $\sigma'=\cone(P\times\{1\})$. Since we assume the answer to Question \ref{Qn:FanoSimplHasTilting} to be positive, we have a simplicial fan $\Sigma'$ with a tilting complex on $\mathcal{X}_{\Sigma'}$ such that the vertices of $P$ correspond to the primitive generators of the rays $\rho\in \sigma'(1)$. By Conjecture \ref{Conj:FanoVanishingII}, there is an NCCR $\Lambda'$ of $R_{\sigma'}=k[(\sigma')^\vee\cap M']$ where $M'$ is the character lattice of $X_{\sigma'}$. Applying Conjecture \ref{Conj:FaceGivesNCCR} gives the desired NCCR $\Lambda$ of $R_\sigma=k[\sigma^\vee\cap M]$.
\end{proof}

\bibliography{Bib}

@inproceedings {BO02,
    AUTHOR = {Bondal, A. and Orlov, D.},
     TITLE = {Derived categories of coherent sheaves},
 BOOKTITLE = {Proceedings of the {I}nternational {C}ongress of
              {M}athematicians, {V}ol. {II} ({B}eijing, 2002)},
     PAGES = {47--56},
 PUBLISHER = {Higher Ed. Press, Beijing},
      YEAR = {2002},
   MRCLASS = {18E30 (14A22 14F05)},
  MRNUMBER = {1957019},
MRREVIEWER = {Bal\'{a}zs Szendr\H{o}i},
}

@article {RS19,
    AUTHOR = {Rennemo, J\o rgen Vold and Segal, Ed},
     TITLE = {Hori-mological projective duality},
   JOURNAL = {Duke Math. J.},
  FJOURNAL = {Duke Mathematical Journal},
    VOLUME = {168},
      YEAR = {2019},
    NUMBER = {11},
     PAGES = {2127--2205},
      ISSN = {0012-7094},
   MRCLASS = {14F08 (16E35 16S38 81T30)},
  MRNUMBER = {3992034},
MRREVIEWER = {Shintarou Yanagida},
       DOI = {10.1215/00127094-2019-0014},
       URL = {https://doi.org/10.1215/00127094-2019-0014},
}

@article {SVdB17,
    AUTHOR = {{\v{S}}penko, \v{S}pela and Van den Bergh, Michel},
     TITLE = {Non-commutative resolutions of quotient singularities for
              reductive groups},
   JOURNAL = {Invent. Math.},
  FJOURNAL = {Inventiones Mathematicae},
    VOLUME = {210},
      YEAR = {2017},
    NUMBER = {1},
     PAGES = {3--67},
      ISSN = {0020-9910},
   MRCLASS = {14A22 (13A50 14L24 16E35)},
  MRNUMBER = {3698338},
MRREVIEWER = {Arvid Siqveland},
       DOI = {10.1007/s00222-017-0723-7},
       URL = {https://doi.org/10.1007/s00222-017-0723-7},
}

@article {Kuz08,
    AUTHOR = {Kuznetsov, Alexander},
     TITLE = {Lefschetz decompositions and categorical resolutions of
              singularities},
   JOURNAL = {Selecta Math. (N.S.)},
  FJOURNAL = {Selecta Mathematica. New Series},
    VOLUME = {13},
      YEAR = {2008},
    NUMBER = {4},
     PAGES = {661--696},
      ISSN = {1022-1824},
   MRCLASS = {18E30 (14E15)},
  MRNUMBER = {2403307},
MRREVIEWER = {Emanuele Macr\'{\i}},
       DOI = {10.1007/s00029-008-0052-1},
       URL = {https://doi.org/10.1007/s00029-008-0052-1},
}

@article {FK18,
    AUTHOR = {Favero, David and Kelly, Tyler L.},
     TITLE = {Fractional {C}alabi-{Y}au categories from {L}andau-{G}inzburg
              models},
   JOURNAL = {Algebr. Geom.},
  FJOURNAL = {Algebraic Geometry},
    VOLUME = {5},
      YEAR = {2018},
    NUMBER = {5},
     PAGES = {596--649},
   MRCLASS = {14F05 (14M25)},
  MRNUMBER = {3847207},
MRREVIEWER = {Sofia Tirabassi},
       DOI = {10.14231/AG-2018-016},
       URL = {https://doi.org/10.14231/AG-2018-016},
}

@incollection {VdB04,
    AUTHOR = {Van den Bergh, Michel},
     TITLE = {Non-commutative crepant resolutions},
 BOOKTITLE = {The legacy of {N}iels {H}enrik {A}bel},
     PAGES = {749--770},
 PUBLISHER = {Springer, Berlin},
      YEAR = {2004},
   MRCLASS = {14A22 (14E15 16S38)},
  MRNUMBER = {2077594},
MRREVIEWER = {Adam Nyman},
}

@article {KM24,
    AUTHOR = {Kelly, Tyler L. and Malter, Aimeric},
     TITLE = {Toric exoflops and categorical resolutions},
   JOURNAL = {J. Inst. Math. Jussieu},
  FJOURNAL = {Journal of the Institute of Mathematics of Jussieu. JIMJ.
              Journal de l'Institut de Math\'{e}matiques de Jussieu},
    VOLUME = {25},
      YEAR = {2026},
    NUMBER = {3},
     PAGES = {1725--1757},
      ISSN = {1474-7480,1475-3030},
   MRCLASS = {14J33 (14F08 14J32 14L24 14M10 14M25)},
  MRNUMBER = {5081817},
       DOI = {10.1017/S1474748026101698},
       URL = {https://doi.org/10.1017/S1474748026101698},
}

@article {Kaw02,
    AUTHOR = {Kawamata, Yujiro},
     TITLE = {{$D$}-equivalence and {$K$}-equivalence},
   JOURNAL = {J. Differential Geom.},
  FJOURNAL = {Journal of Differential Geometry},
    VOLUME = {61},
      YEAR = {2002},
    NUMBER = {1},
     PAGES = {147--171},
      ISSN = {0022-040X},
   MRCLASS = {14F05 (14E05 14E30 18E30)},
  MRNUMBER = {1949787},
MRREVIEWER = {Ana Jerem\'{\i}as L\'{o}pez},
       URL = {http://projecteuclid.org/euclid.jdg/1090351323},
}

@article {VdB3Dflops,
    AUTHOR = {Van den Bergh, Michel},
     TITLE = {Three-dimensional flops and noncommutative rings},
   JOURNAL = {Duke Math. J.},
  FJOURNAL = {Duke Mathematical Journal},
    VOLUME = {122},
      YEAR = {2004},
    NUMBER = {3},
     PAGES = {423--455},
      ISSN = {0012-7094},
   MRCLASS = {14E30 (14A22 16S38 18E30)},
  MRNUMBER = {2057015},
MRREVIEWER = {Christopher J. Pappacena},
       DOI = {10.1215/S0012-7094-04-12231-6},
       URL = {https://doi.org/10.1215/S0012-7094-04-12231-6},
}

@incollection {Leu12,
    AUTHOR = {Leuschke, Graham J.},
     TITLE = {Non-commutative crepant resolutions: scenes from categorical
              geometry},
 BOOKTITLE = {Progress in commutative algebra 1},
     PAGES = {293--361},
 PUBLISHER = {de Gruyter, Berlin},
      YEAR = {2012},
   MRCLASS = {14A22 (14E15 14F05 16E35 16S38)},
  MRNUMBER = {2932589},
MRREVIEWER = {Adam Nyman},
}

@incollection {VdB23,
    AUTHOR = {Van den Bergh, Michel},
     TITLE = {Noncommutative crepant resolutions, an overview},
 BOOKTITLE = {I{CM}---{I}nternational {C}ongress of {M}athematicians. {V}ol.
              2. {P}lenary lectures},
     PAGES = {1354--1391},
 PUBLISHER = {EMS Press, Berlin},
      YEAR = {[2023] \copyright 2023},
   MRCLASS = {14A22 (14E30 18G80)},
  MRNUMBER = {4680283},
}

@article {BFK19,
    AUTHOR = {Ballard, Matthew and Favero, David and Katzarkov, Ludmil},
     TITLE = {Variation of geometric invariant theory quotients and derived
              categories},
   JOURNAL = {J. Reine Angew. Math.},
  FJOURNAL = {Journal f\"{u}r die Reine und Angewandte Mathematik. [Crelle's
              Journal]},
    VOLUME = {746},
      YEAR = {2019},
     PAGES = {235--303},
      ISSN = {0075-4102},
   MRCLASS = {14F05 (14J33 14L24)},
  MRNUMBER = {3895631},
MRREVIEWER = {Richard P. Thomas},
       DOI = {10.1515/crelle-2015-0096},
       URL = {https://doi.org/10.1515/crelle-2015-0096},
}

@article {Hirano,
    AUTHOR = {Hirano, Yuki},
     TITLE = {Derived {K}n\"{o}rrer periodicity and {O}rlov's theorem for gauged
              {L}andau-{G}inzburg models},
   JOURNAL = {Compos. Math.},
  FJOURNAL = {Compositio Mathematica},
    VOLUME = {153},
      YEAR = {2017},
    NUMBER = {5},
     PAGES = {973--1007},
      ISSN = {0010-437X},
   MRCLASS = {14F05 (14J33 14L24 18E30)},
  MRNUMBER = {3631231},
MRREVIEWER = {Tyler L. Kelly},
       DOI = {10.1112/S0010437X16008344},
       URL = {https://doi.org/10.1112/S0010437X16008344},
}

@article {FMS19,
    AUTHOR = {Faber, Eleonore and Muller, Greg and Smith, Karen E.},
     TITLE = {Non-commutative resolutions of toric varieties},
   JOURNAL = {Adv. Math.},
  FJOURNAL = {Advances in Mathematics},
    VOLUME = {351},
      YEAR = {2019},
     PAGES = {236--274},
      ISSN = {0001-8708},
   MRCLASS = {13C14 (13A35 14A22 14M25 16E10 16S32)},
  MRNUMBER = {3951483},
MRREVIEWER = {Graham J. Leuschke},
       DOI = {10.1016/j.aim.2019.04.021},
       URL = {https://doi.org/10.1016/j.aim.2019.04.021},
}

@article {Bei78,
    AUTHOR = {Beilinson, A. A.},
     TITLE = {Coherent sheaves on {${\bf P}\sp{n}$} and problems in linear
              algebra},
   JOURNAL = {Funktsional. Anal. i Prilozhen.},
  FJOURNAL = {Akademiya Nauk SSSR. Funktsional\cprime ny\u{\i} Analiz i ego
              Prilozheniya},
    VOLUME = {12},
      YEAR = {1978},
    NUMBER = {3},
     PAGES = {68--69},
      ISSN = {0374-1990},
   MRCLASS = {14F05 (18E30 81C99)},
  MRNUMBER = {509388},
MRREVIEWER = {P. E. Newstead},
}

@article {KK20,
    AUTHOR = {Keller, Bernhard and Krause, Henning},
     TITLE = {Tilting preserves finite global dimension},
   JOURNAL = {C. R. Math. Acad. Sci. Paris},
  FJOURNAL = {Comptes Rendus Math\'{e}matique. Acad\'{e}mie des Sciences. Paris},
    VOLUME = {358},
      YEAR = {2020},
    NUMBER = {5},
     PAGES = {563--570},
      ISSN = {1631-073X},
   MRCLASS = {16E10 (16D90 16E05 18E10 18G80)},
  MRNUMBER = {4149855},
MRREVIEWER = {Jian Wang},
       DOI = {10.5802/crmath.72},
       URL = {https://doi.org/10.5802/crmath.72},
}

@article {IW14a,
    AUTHOR = {Iyama, Osamu and Wemyss, Michael},
     TITLE = {Maximal modifications and {A}uslander-{R}eiten duality for
              non-isolated singularities},
   JOURNAL = {Invent. Math.},
  FJOURNAL = {Inventiones Mathematicae},
    VOLUME = {197},
      YEAR = {2014},
    NUMBER = {3},
     PAGES = {521--586},
      ISSN = {0020-9910},
   MRCLASS = {13C14 (13D09 16G70)},
  MRNUMBER = {3251829},
MRREVIEWER = {Siamak Yassemi},
       DOI = {10.1007/s00222-013-0491-y},
       URL = {https://doi.org/10.1007/s00222-013-0491-y},
}

@article {Broomhead,
    AUTHOR = {Broomhead, Nathan},
     TITLE = {Dimer models and {C}alabi-{Y}au algebras},
   JOURNAL = {Mem. Amer. Math. Soc.},
  FJOURNAL = {Memoirs of the American Mathematical Society},
    VOLUME = {215},
      YEAR = {2012},
    NUMBER = {1011},
     PAGES = {viii+86},
      ISSN = {0065-9266},
      ISBN = {978-0-8218-5308-5},
   MRCLASS = {82B20 (14M25)},
  MRNUMBER = {2908565},
       DOI = {10.1090/S0065-9266-2011-00617-9},
       URL = {https://doi.org/10.1090/S0065-9266-2011-00617-9},
}

@book {CLS,
    AUTHOR = {Cox, David A. and Little, John B. and Schenck, Henry K.},
     TITLE = {Toric varieties},
    SERIES = {Graduate Studies in Mathematics},
    VOLUME = {124},
 PUBLISHER = {American Mathematical Society, Providence, RI},
      YEAR = {2011},
     PAGES = {xxiv+841},
      ISBN = {978-0-8218-4819-7},
   MRCLASS = {14M25 (05A15 05E45 52B12)},
  MRNUMBER = {2810322},
MRREVIEWER = {Ivan Arzhantsev},
       DOI = {10.1090/gsm/124},
       URL = {https://doi.org/10.1090/gsm/124},
}

@article {BFK14,
    AUTHOR = {Ballard, Matthew and Favero, David and Katzarkov, Ludmil},
     TITLE = {A category of kernels for equivariant factorizations and its
              implications for {H}odge theory},
   JOURNAL = {Publ. Math. Inst. Hautes \'{E}tudes Sci.},
  FJOURNAL = {Publications Math\'{e}matiques. Institut de Hautes \'{E}tudes
              Scientifiques},
    VOLUME = {120},
      YEAR = {2014},
     PAGES = {1--111},
      ISSN = {0073-8301},
   MRCLASS = {14F05 (14C30 18E30 18G55)},
  MRNUMBER = {3270588},
MRREVIEWER = {Pawel Sosna},
       DOI = {10.1007/s10240-013-0059-9},
       URL = {https://doi.org/10.1007/s10240-013-0059-9},
}

@article {Orlov92,
    AUTHOR = {Orlov, D. O.},
     TITLE = {Projective bundles, monoidal transformations, and derived
              categories of coherent sheaves},
   JOURNAL = {Izv. Ross. Akad. Nauk Ser. Mat.},
  FJOURNAL = {Izvestiya Rossiiskoi Akademii Nauk. Seriya Matematicheskaya},
    VOLUME = {56},
      YEAR = {1992},
    NUMBER = {4},
     PAGES = {852--862},
      ISSN = {1607-0046},
   MRCLASS = {14F05 (18E30 18F20)},
  MRNUMBER = {1208153},
MRREVIEWER = {Krzysztof Jaczewski},
       DOI = {10.1070/IM1993v041n01ABEH002182},
       URL = {https://doi.org/10.1070/IM1993v041n01ABEH002182},
}

@article {SVdBtoricII,
    AUTHOR = {{\v{S}}penko, \v{S}pela and Van den Bergh, Michel},
     TITLE = {Non-commutative crepant resolutions for some toric
              singularities. {II}},
   JOURNAL = {J. Noncommut. Geom.},
  FJOURNAL = {Journal of Noncommutative Geometry},
    VOLUME = {14},
      YEAR = {2020},
    NUMBER = {1},
     PAGES = {73--103},
      ISSN = {1661-6952},
   MRCLASS = {14M25 (14A22 16S38 32S45)},
  MRNUMBER = {4107511},
MRREVIEWER = {Pieter Belmans},
       DOI = {10.4171/JNCG/359},
       URL = {https://doi.org/10.4171/JNCG/359},
}

@article {Nova18,
    AUTHOR = {Novakovi\'{c}, Sa\v{s}a},
     TITLE = {Tilting objects on some global quotient stacks},
   JOURNAL = {J. Commut. Algebra},
  FJOURNAL = {Journal of Commutative Algebra},
    VOLUME = {10},
      YEAR = {2018},
    NUMBER = {1},
     PAGES = {107--137},
      ISSN = {1939-0807},
   MRCLASS = {14F05 (14A20)},
  MRNUMBER = {3804849},
MRREVIEWER = {Simone Marchesi},
       DOI = {10.1216/JCA-2018-10-1-107},
       URL = {https://doi.org/10.1216/JCA-2018-10-1-107},
}

@incollection {BN07,
    AUTHOR = {Batyrev, Victor and Nill, Benjamin},
     TITLE = {Combinatorial aspects of mirror symmetry},
 BOOKTITLE = {Integer points in polyhedra---geometry, number theory,
              representation theory, algebra, optimization, statistics},
    SERIES = {Contemp. Math.},
    VOLUME = {452},
     PAGES = {35--66},
 PUBLISHER = {Amer. Math. Soc., Providence, RI},
      YEAR = {2008},
   MRCLASS = {14J32 (14M25 52B20)},
  MRNUMBER = {2405763},
MRREVIEWER = {Julianna Tymoczko},
       DOI = {10.1090/conm/452/08770},
       URL = {https://doi.org/10.1090/conm/452/08770},
}

@article {IW14b,
    AUTHOR = {Iyama, Osamu and Wemyss, Michael},
     TITLE = {Singular derived categories of {$\Bbb{Q}$}-factorial
              terminalizations and maximal modification algebras},
   JOURNAL = {Adv. Math.},
  FJOURNAL = {Advances in Mathematics},
    VOLUME = {261},
      YEAR = {2014},
     PAGES = {85--121},
      ISSN = {0001-8708},
   MRCLASS = {14F05 (13H10 14M05 18E30)},
  MRNUMBER = {3213296},
MRREVIEWER = {P. E. Newstead},
       DOI = {10.1016/j.aim.2014.04.001},
       URL = {https://doi.org/10.1016/j.aim.2014.04.001},
}

@article {RZ03,
    AUTHOR = {Rouquier, Rapha\"{e}l and Zimmermann, Alexander},
     TITLE = {Picard groups for derived module categories},
   JOURNAL = {Proc. London Math. Soc. (3)},
  FJOURNAL = {Proceedings of the London Mathematical Society. Third Series},
    VOLUME = {87},
      YEAR = {2003},
    NUMBER = {1},
     PAGES = {197--225},
      ISSN = {0024-6115},
   MRCLASS = {16D90 (18E30 20F36 20J05)},
  MRNUMBER = {1978574},
MRREVIEWER = {Henning Krause},
       DOI = {10.1112/S0024611503014059},
       URL = {https://doi.org/10.1112/S0024611503014059},
}

@misc{BBB+,
      title={King's Conjecture and Birational Geometry, arXiv:2501.00130}, 
      author={Matthew R. Ballard and Christine Berkesch and Michael K. Brown and Lauren Cranton Heller and Daniel Erman and David Favero and Sheel Ganatra and Andrew Hanlon and Jesse Huang},
      year={2025},
      eprint={2501.00130},
      archivePrefix={arXiv},
      primaryClass={math.AG},
      url={https://arxiv.org/abs/2501.00130}, 
}

@article {Gul,
    AUTHOR = {Gulotta, Daniel R.},
     TITLE = {Properly ordered dimers, {$R$}-charges, and an efficient
              inverse algorithm},
   JOURNAL = {J. High Energy Phys.},
  FJOURNAL = {Journal of High Energy Physics. A SISSA Journal},
      YEAR = {2008},
    NUMBER = {10},
     PAGES = {014, 31},
      ISSN = {1126-6708},
   MRCLASS = {81T30 (81T13 81T50)},
  MRNUMBER = {2453031},
       DOI = {10.1088/1126-6708/2008/10/014},
       URL = {https://doi.org/10.1088/1126-6708/2008/10/014},
}

@article {IU16,
    AUTHOR = {Ishii, Akira and Ueda, Kazushi},
     TITLE = {Dimer models and crepant resolutions},
   JOURNAL = {Hokkaido Math. J.},
  FJOURNAL = {Hokkaido Mathematical Journal},
    VOLUME = {45},
      YEAR = {2016},
    NUMBER = {1},
     PAGES = {1--42},
      ISSN = {0385-4035},
   MRCLASS = {14L24 (14D20 14J32 14M25)},
  MRNUMBER = {3532120},
MRREVIEWER = {Piotr Su\l kowski},
       DOI = {10.14492/hokmj/1470080746},
       URL = {https://doi.org/10.14492/hokmj/1470080746},
}

@article {IU15,
    AUTHOR = {Ishii, Akira and Ueda, Kazushi},
     TITLE = {Dimer models and the special {M}c{K}ay correspondence},
   JOURNAL = {Geom. Topol.},
  FJOURNAL = {Geometry \& Topology},
    VOLUME = {19},
      YEAR = {2015},
    NUMBER = {6},
     PAGES = {3405--3466},
      ISSN = {1465-3060},
   MRCLASS = {14F05 (14D20 14E16 16G10 82B20)},
  MRNUMBER = {3447107},
MRREVIEWER = {Howard Nuer},
       DOI = {10.2140/gt.2015.19.3405},
       URL = {https://doi.org/10.2140/gt.2015.19.3405},
}

@incollection {IU08,
    AUTHOR = {Ishii, Akira and Ueda, Kazushi},
     TITLE = {On moduli spaces of quiver representations associated with
              dimer models},
 BOOKTITLE = {Higher dimensional algebraic varieties and vector bundles},
    SERIES = {RIMS K\^{o}ky\^{u}roku Bessatsu, B9},
     PAGES = {127--141},
 PUBLISHER = {Res. Inst. Math. Sci. (RIMS), Kyoto},
      YEAR = {2008},
   MRCLASS = {16G20 (14D21 81T30 82B20)},
  MRNUMBER = {2509696},
}

@article {Isik,
    AUTHOR = {Isik, Mehmet Umut},
     TITLE = {Equivalence of the derived category of a variety with a
              singularity category},
   JOURNAL = {Int. Math. Res. Not. IMRN},
  FJOURNAL = {International Mathematics Research Notices. IMRN},
      YEAR = {2013},
    NUMBER = {12},
     PAGES = {2787--2808},
      ISSN = {1073-7928},
   MRCLASS = {14F05 (14B05)},
  MRNUMBER = {3071664},
MRREVIEWER = {Pawel Sosna},
       DOI = {10.1093/imrn/rns125},
       URL = {https://doi.org/10.1093/imrn/rns125},
}

@article {Shipman,
    AUTHOR = {Shipman, Ian},
     TITLE = {A geometric approach to {O}rlov's theorem},
   JOURNAL = {Compos. Math.},
  FJOURNAL = {Compositio Mathematica},
    VOLUME = {148},
      YEAR = {2012},
    NUMBER = {5},
     PAGES = {1365--1389},
      ISSN = {0010-437X},
   MRCLASS = {14F05 (14D23 14J32 14J70 81T30)},
  MRNUMBER = {2982435},
MRREVIEWER = {Antony Maciocia},
       DOI = {10.1112/S0010437X12000255},
       URL = {https://doi.org/10.1112/S0010437X12000255},
}

@article {BH09,
    AUTHOR = {Borisov, Lev and Hua, Zheng},
     TITLE = {On the conjecture of {K}ing for smooth toric
              {D}eligne-{M}umford stacks},
   JOURNAL = {Adv. Math.},
  FJOURNAL = {Advances in Mathematics},
    VOLUME = {221},
      YEAR = {2009},
    NUMBER = {1},
     PAGES = {277--301},
      ISSN = {0001-8708},
   MRCLASS = {14M25 (14A20 14F05)},
  MRNUMBER = {2509327},
MRREVIEWER = {T. Oda},
       DOI = {10.1016/j.aim.2008.11.017},
       URL = {https://doi.org/10.1016/j.aim.2008.11.017},
}

@book {HandbookDiscrete,
     TITLE = {Handbook of discrete and computational geometry},
    SERIES = {Discrete Mathematics and its Applications (Boca Raton)},
    EDITOR = {Goodman, Jacob E. and O'Rourke, Joseph and T\'{o}th, Csaba D.},
      NOTE = {Third edition of [ MR1730156]},
 PUBLISHER = {CRC Press, Boca Raton, FL},
      YEAR = {2018},
     PAGES = {xxi+1927},
      ISBN = {978-1-4987-1139-5},
   MRCLASS = {52-00 (00A20 05B25 65D18 68-06)},
  MRNUMBER = {3793131},
MRREVIEWER = {Egon Schulte},
}

@article {IR08,
    AUTHOR = {Iyama, Osamu and Reiten, Idun},
     TITLE = {Fomin-{Z}elevinsky mutation and tilting modules over
              {C}alabi-{Y}au algebras},
   JOURNAL = {Amer. J. Math.},
  FJOURNAL = {American Journal of Mathematics},
    VOLUME = {130},
      YEAR = {2008},
    NUMBER = {4},
     PAGES = {1087--1149},
      ISSN = {0002-9327},
   MRCLASS = {16G30 (16G20 16S38)},
  MRNUMBER = {2427009},
MRREVIEWER = {Hugh Ross Thomas},
       DOI = {10.1353/ajm.0.0011},
       URL = {https://doi.org/10.1353/ajm.0.0011},
}

@article {HM06,
    AUTHOR = {Haase, Christian and Melnikov, Ilarion V.},
     TITLE = {The reflexive dimension of a lattice polytope},
   JOURNAL = {Ann. Comb.},
  FJOURNAL = {Annals of Combinatorics},
    VOLUME = {10},
      YEAR = {2006},
    NUMBER = {2},
     PAGES = {211--217},
      ISSN = {0218-0006},
   MRCLASS = {52B20 (14M25)},
  MRNUMBER = {2258235},
MRREVIEWER = {Lucia Maria Marino},
       DOI = {10.1007/s00026-006-0283-9},
       URL = {https://doi.org/10.1007/s00026-006-0283-9},
}

@article {Casa06,
    AUTHOR = {Casagrande, Cinzia},
     TITLE = {The number of vertices of a {F}ano polytope},
   JOURNAL = {Ann. Inst. Fourier (Grenoble)},
  FJOURNAL = {Universit\'{e} de Grenoble. Annales de l'Institut Fourier},
    VOLUME = {56},
      YEAR = {2006},
    NUMBER = {1},
     PAGES = {121--130},
      ISSN = {0373-0956},
   MRCLASS = {14M25 (14J45 52B20)},
  MRNUMBER = {2228683},
MRREVIEWER = {Boris J. Kazarnovski\u{\i}},
       DOI = {10.5802/aif.2175},
       URL = {https://doi.org/10.5802/aif.2175},
}

@article {BS10,
    AUTHOR = {Bridgeland, Tom and Stern, David},
     TITLE = {Helices on del {P}ezzo surfaces and tilting {C}alabi-{Y}au
              algebras},
   JOURNAL = {Adv. Math.},
  FJOURNAL = {Advances in Mathematics},
    VOLUME = {224},
      YEAR = {2010},
    NUMBER = {4},
     PAGES = {1672--1716},
      ISSN = {0001-8708},
   MRCLASS = {14F05 (14J26 16E35 16G20)},
  MRNUMBER = {2646308},
MRREVIEWER = {Johannes Walcher},
       DOI = {10.1016/j.aim.2010.01.018},
       URL = {https://doi.org/10.1016/j.aim.2010.01.018},
}

@misc{IdPKW,
AUTHOR={Iyama, Osamu and de la Pena, Jose Antonio and Krause, Henning and Wemyss, Michael},
TITLE= {Tilting Theory, Singularity Categories, \& Noncommutative Resolutions, https://www.birs.ca/cmo-workshops/2019/19w5161/report19w5161.pdf},
YEAR ={2019},
URL={https://www.birs.ca/cmo-workshops/2019/19w5161/report19w5161.pdf},
}

@misc{DHKR25,
AUTHOR={Donovan, Will and Hara, Wahei and Kapustka, Micha\l{} and Rampazzo, Marco},
TITLE={Window categories for a simple 9-fold flop of Grassmannian type, arXiV:2510.06184},
YEAR={2025},
URL={https://arxiv.org/abs/2510.06184},
eprint={2510.06184},
archivePrefix={arXiv},
primaryClass={math.AG}
}

@article {KN25,
    AUTHOR = {Krug, Andreas and Nikolov, Erik},
     TITLE = {Endomorphism algebras of equivariant exceptional collections},
   JOURNAL = {Algebr. Represent. Theory},
  FJOURNAL = {Algebras and Representation Theory},
    VOLUME = {28},
      YEAR = {2025},
    NUMBER = {1},
     PAGES = {193--210},
      ISSN = {1386-923X,1572-9079},
   MRCLASS = {18G80 (14F08 14L30 16E35 16G10 16G20)},
  MRNUMBER = {4875693},
MRREVIEWER = {Xiao-Wu\ Chen},
       DOI = {10.1007/s10468-025-10313-0},
       URL = {https://doi.org/10.1007/s10468-025-10313-0},
}
\end{document}